\documentclass[11pt]{article}
\usepackage [dvips]{graphics}
\usepackage[centertags]{amsmath}
\usepackage{amsfonts}
\usepackage{amssymb}
\usepackage{amsthm}
\usepackage{amsmath,amscd}
\usepackage{newlfont}
\usepackage{stmaryrd}
\usepackage[]{titletoc}
\usepackage{lipsum}
\usepackage{soul}
\usepackage{mathrsfs}
\usepackage{xypic}
\newlength{\defbaselineskip}
\newcommand{\setlinespacing}[1]%
           {\setlength{\baselineskip}{#1 \defbaselineskip}}

\theoremstyle{plain}
\newtheorem{thm}{Theorem}[section]
\newtheorem{defn}[thm]{Definition}
\newtheorem{cor}[thm]{Corollary}
\newtheorem{lem}[thm]{Lemma}

\newtheorem{exam}[thm]{Example}
\newtheorem{rem}[thm]{Remark}

\newtheorem*{arv}{Arveson-Douglas Conjecture}
\newtheorem*{geo}{Geometric Arveson-Douglas Conjecture}

\let\sss=\scriptscriptstyle

\titlecontents{section}[2em]{}
    {\thecontentslabel\quad}{}{{}}

\titlecontents{subsection}[2em]{}
    {\thecontentslabel\quad}{}{{}}

\newcommand{\bn}{\mathbb{B}_n}
\newcommand{\cn}{\mathbb{C}^n}
\newcommand{\toe}{\mathcal{T}(L^{\infty})}
\newcommand{\ber}{L_a^{2}(\mathbb{B}_n)}
\newcommand{\pbn}{\partial\mathbb{B}_n}
\newcommand{\clb}{\overline{\mathbb{B}_n}}
\newcommand{\K}{\mathcal{K}(L_a^2(\mathbb{B}_n))}
\newcommand{\Tmus}{T_{\mu_s}}
\newcommand{\tTmus}{\tilde{T}_{\mu_s}}

\newcommand{\Ps}{\mathcal{P}(\mu_s)}

\newcommand\blfootnote[1]{%
  \begingroup
  \renewcommand\thefootnote{}\footnote{#1}%
  \addtocounter{footnote}{-1}%
  \endgroup
}

\makeatletter\@addtoreset{equation}{section} \makeatother
\begin{document}
\title {Geometric Arveson-Douglas Conjecture and Holomorphic Extension}
\author{Ronald G. Douglas, Yi Wang}

\date{}
\maketitle
\blfootnote{


The second author would like to thank for the hospitality of Taxes A\&M University during her visit here. She also want to thank for the financial support from the program of China Scholarships Council(CSC).
}
\begin{abstract}
In this paper we introduce techniques from complex harmonic analysis to prove a weaker version of the Geometric Arveson-Douglas Conjecture for complex analytic subsets that is smooth on the boundary of the unit ball and intersects transversally with it. In fact, we prove that the projection operator onto the corresponding quotient module is in the Toeplitz algebra $\toe$, which implies the essential normality of the quotient module. Combining some other techniques we actually obtain the $p$-essential normality for $p>2d$, where $d$ is the complex dimension of the analytic subset.  Finally, we show that our results apply for the closure of a radical polynomial ideal $I$ whose zero variety satisfies the above conditions. A key technique is defining a right inverse operator of the restriction map from the unit ball to the analytic subset generalizing the result of Beatrous's paper \cite{Beatrous}.
\end{abstract}


\section{Introduction}

A complex Hilbert space $\mathcal{H}$ is called a Hilbert module (over the polynomial ring $\mathbb{C}[z_1,\cdots,z_n]$) if for every $p\in\mathbb{C}[z_1,\cdots,z_n]$ there is a bounded linear operator $M_p$ on $\mathcal{H}$ and the map $\mathbb{C}[z_1,\cdots,z_n]\to B(\mathcal{H}),~p\to M_p$ is an algebra homomorphism. Examples include the Bergman module, Hardy module and the Drury-Arveson module over $\bn$, the unit ball of $\mathbb{C}^n$. A submodule $P\subset\mathcal{H}$ is a Hilbert subspace of $\mathcal{H}$ that is closed under the module multiplications $M_p$. Let $Q$ be the orthogonal complement of $P$ in $\mathcal{H}$. Then $Q$ is the quotient Hilbert module with the homomorphism taking $z_i$ to
the compression of $M_{z_i}$ to $Q$. A Hilbert module $\mathcal{H}$ is said to be essentially normal ($p$-essentially normal ) if the commutators $[M_{z_i}, M_{z_j}^*]$ belong to the compact operators $\mathcal{K}(\mathcal{H})$ ( Schatten $p$ class $\mathcal{S}^p$), for any $1\leq i, j\leq n$.
Hilbert modules play an important role in multivariate operator theory.
In his paper \cite{Arv Dirac}, Arveson made a conjecture, which was refined by the first author in \cite{Dou index} to the following form:
\begin{arv}
Let $\mathcal{H}$ be one of $\ber$, $H^2(\bn)$ and $H^2_n(\bn)$. Assume $I$ is a homogeneous ideal of $\mathbb{C}[z_1,\cdots,z_n]$
and $P$ is the closure of $I$ in $\mathcal{H}$. Then for all $p>\dim Z(I)$, the quotient module $Q=P^{\perp}$ is $p$-essentially normal. Here $Z(I)$ is the zero set of $I$ and $\dim Z(I)$ denotes the complex dimension of $Z(I)$.
\end{arv}
The Arveson-Douglas Conjecture has been proved under various conditions. Arveson \cite{Arv p summable} provd the case when $I$ is generated by monomials; Guo \cite{Guo} proved the case when $n=2$; Guo and Wang \cite{Guo Wang} proved the case when $I$ is a principal homogenous ideal, when the complex dimension of $Z(I)\leq1$, and for any homogenous ideal $I$ when the dimension $n\leq 3$; they \cite{Guo Wang quosi} also proved the case when $I$ is a quasi-homogeneous ideal and $n=2$; the first author and Wang \cite{Douglas Wang} proved the case for $\ber$ when $I$ is any principal ideal; in \cite{Douglas Sarkar}, the first author and Sarkar reduced the quasi-homogenous case to the homogenous ones; Shalit \cite{Shalit} proved the case when the submodule possesses the stable division property.

In many cases( cf. \cite{DYT}), especially when $I$ is radical, one can prove that $[I]$, the closure of the ideal $I$, consists of all the holomorphic functions in $\mathcal{H}$ that vanish on $V(I)$. Here
$$
V(I)=\{z\in\bn: p(z)=0, \forall p\in I\}.
$$
In these cases, we can relate the submodule $[I]$ to its zero variety. The following conjecture first appeared in \cite{Sha Ken} and is a reformulation of the Arveson-Douglas Conjecture in these special cases.
\begin{geo}
Let $M$ be a homogeneous variety in $\bn$. Let
$$
P=\{f\in\mathcal{H}: f|_M=0\}
$$
and $Q=P^{\perp}$. Then the quotient module $Q$ is $p$-essentially normal for every $p>\dim M$. Here $\mathcal{H}$ is an analytic Hilbert module(cf. \cite{Guo Chen}), such as the Bergman module $\ber$, the Hardy module $H^2(\bn)$ or the Drury-Arveson module $H^2_n(\bn)$.
\end{geo}
There is also a weaker form of the two conjectures: to replace $p$-essential normality with essential normality.
Shalit and Kennedy \cite {Sha Ken} proved the case when $M$ can be decomposed into varieties having ``good'' properties; Engli\v{s} and Eschmeier \cite{Englis} proved the case when $M$ is a homogeneous subvariety that is smooth away from the origin and the case when $M$ is a (not necessarily homogeneous) smooth submanifold that intersects $\pbn$ transversally; the first author, Tang and Yu \cite{DYT} proved the weaker form under the assumption that $M$ is a (not necessarily homogenous) complete intersection space that intersects $\pbn$ transversally and has no singular point on $\pbn$.

 In this paper we mainly consider the weaker form of the Geometric Arveson-Douglas Conjecture on a not necessarily homogeneous variety. The methods of this paper differ from those in \cite{DYT} and \cite{Englis}. We obtain essential normality by proving that the projection operator onto the submodule is in the Toeplitz algebra. This approach is new and allows us to analyse the conjecture using tools from complex harmonic analysis. Part of our ideas come from Su\'{a}rez's paper \cite{Suarez07} and the first author, Tang and Yu's paper \cite{DYT}.

 In section \ref{mainsec}, we use some ideas in Su\'{a}rez's paper \cite{Suarez07} to show that the existence of a positive measure on $M$ that defines an equivalent norm on the quotient module will imply the essential normality of the quotient module.
 \begin{thm}(Theorem \ref{main})
 If there exists a positive measure~$\mu$~on~$M$~such that the~$L_{\mu}^2$~norm and Bergman norm are equivalent on the quotient module~$Q$, i.e., $\exists C, c>0$ such that $\forall f\in Q$,
$$c\|f\|^2\leq\int_M|f(w)|^2d\mu(w)\leq C\|f\|^2,$$
then the quotient module~$Q$~is essentially normal.
 \end{thm}
 Equivalently, we show that the existence of an extension map which is a right inverse of the restriction map from $\bn$ to the zero variety(cf. \cite{Beatrous} and section 4.2 in \cite{DYT}) implies the essential normality of the quotient module.(See the remark in the end of section \ref{mainsec}). Moreover, as a consequence of the proof, we have the following interesting corollary:
 \begin{cor}(Corollary \ref{toe2})
 Assume the same hypotheses as in Theorem~\ref{main}, then the projection operators~$P$~and $Q$ are in the Toeplitz algevra~$\toe$.
 \end{cor}

 In section \ref{manifoldsec}, we combine techniques from complex harmonic analysis with those from classical operator theory to prove that such a measure always exists under some conditions. More precisely, we have the following theorem:
\begin{thm}{(Theorem \ref{manifold})}
Suppose~$\tilde{M}$~is a complex analytic subset of an open neighborhood of~$\clb$~satisfying the following conditions:
\begin{itemize}
\item[(1)] $\tilde{M}$~intersects~$\pbn$~transversally.
\item[(2)] $\tilde{M}$~has no singular point on~$\pbn$.
\end{itemize}
Let~$M=\tilde{M}\cap\bn$ and let~$P=\{f\in\ber: f|_M=0\}$. Then the submodule~$P$~is essentially normal.
\end{thm}
Applying the previous corollary to this case, we show that the corresponding projection operator is in the Toeplitz $C^*$-algebra.
\begin{cor}{(Corollary \ref{toe})}
Suppose $M$ satisfies the properties of Theorem \ref{manifold}, then the projection operator onto $P$ is in the Toeplitz algebra $\toe$.
\end{cor}
In connection with the theory of holomorphic extension, our approach proves the existence of an extension operator from the analytic subset to the unit ball that is bounded in the $L^2$ norm. This result extends the result in \cite{Beatrous}. We conclude section \ref{manifoldsec} with a remark related to this topic, including the following corollary.
\begin{cor}{(Corollary \ref{extension})}
Suppose $\tilde{M}$ is a $d$-dimensional complex analytic subset of an open neighborhood of $\clb$ which intersects $\pbn$ transversally and has no singular point on $\pbn$. Let $M=\tilde{M}\cap\bn$, $\mu=(1-|w|^2)^{n-d}dv_d$ on $M$ and $\mathcal{P}$ be the closed subspace of $L^2(\mu)$ generated by analytic polynomials. Then $\mathcal{P}$ is exactly the range of the restriction operator
$$
R: \ber\to L^2(\mu), f\mapsto f|_M.
$$
Equivalently, there is an extension operator
$$
E: \mathcal{P}\to\ber,
$$
which is a right inverse of $R$.
\end{cor}

Section \ref{secpess} is an attempt towards the full Geometric Arveson-Douglas Conjecture. We combine our methods and the methods in \cite{hankel} to show:
\begin{thm}{(Theorem \ref{2d ess nor})}
Under the assumptions of Theorem \ref{manifold}, the quotient module $Q$ is $p$-essentially normal for all $p>2d$.
\end{thm}
In fact, we prove that the range space $\mathcal{P}$ is $p$-essentially normal for $p>d$. Although we didn't get the $p$($p>d$)-essential normality of $Q$, we still think it's possible to prove it using similar methods.

In section \ref{arvsec}, we relate the Geometric Arveson-Douglas Conjecture with the Arveson-Douglas Conjecture. We prove the following theorem by extending the proof in the appendix of \cite{DYT} to our case.
\begin{thm}(Theorem \ref{arv})
Suppose $I\subset\mathbb{C}[z_1,\cdots,z_n]$ is a radical polynomial ideal. If $\tilde{M}:=Z(I)$ has no singular point on $\pbn$ and intersects $\pbn$ transversally, then $[I]=\{f\in\ber: f|_M=0\}$. Here $M=\tilde{M}\cap\bn$. As a consequence, the quotient module $Q=[I]^{\perp}$ is $p$-essentially normal, $p>2d$. Moreover, the projection operator onto $Q$ is in the Toeplitz algebra $\toe$.
\end{thm}

 Our result generalizes the essential normality results of those in \cite{Englis} and \cite{DYT}. Moreover, it was not shown in the previous two papers that the projection operators are in the Toeplitz algebra. This could be useful, for example, in the study of joint invertibility and joint Fredholmness of Toeplitz operators on the submodules and quotient modules.

 This new method is simple and calculable. We believe that there is more to be discovered. We hope to obtain generalizations to $p$-essential normality( $p>d$) and varieties satisfying other assumptions in the future.

\section{Preliminaries}
Let $\bn$ be the unit ball of $\cn$. The Bergman space $\ber$ is defined as all holomorphic functions on $\bn$ that is square integrable.
$$
\ber=\{f:\bn\to\mathbb{C}|~f\mbox{ is holomorphic and }\int_{\bn}|f(w)|^2dv_n(w)<\infty\}.
$$
Here $v_n$ is the normalized volume measure on $\bn$. For $g\in L^{\infty}(\bn)$, the Toeplitz operator is defined by
$$
T_g:\ber\to\ber, f\mapsto P_{\ber}(gf)
$$
where $P_{\ber}$ is the projection operator from $L^2(\bn)$ to $\ber$. The Toeplitz algebra $\toe$ is the $C^*$ subalgebra of $B(\ber)$ generated by $T_g$, $g\in L^{\infty}(\bn)$. The space $\ber$ is a Hilbert module with module multiplications $T_p, p\in \mathbb{C}[z_1,\cdots,z_n]$. By definition, a submodule of $\ber$ is a closed subspace that is invariant under $T_{z_i}, i=1,\cdots,n$.

In the rest of this paper, $M$~will always mean a subset of~$\bn$, $P$~means the submodule of~$\ber$~consisting of all functions that vanish on~$M$~and~$Q=P^{\perp}$. We also use~$P$, $Q$~to denote the corresponding projection operators to each space. Note that under this setting,
$$Q=\overline{span}\{k_z: z\in M\}.$$

In order to prove the essential normality of the submodule~$P$, we need the following lemma, which is well known.

\begin{lem}\label{enormal}
Suppose $P$ is a submodule of the Bergman module $\ber$, then the following are equivalent:
\begin{itemize}
\item[(1)] The submodule $P$ is essential normal.
\item[(2)] The quotient module $Q$ is essential normal.
\item[(3)] $[P, M_{z_i}]\in \K$, $i=1,\cdots,n$.
\item[(4)] $[Q, M_{z_i}]\in \K$, $i=1,\cdots,n$.
\end{itemize}
\end{lem}

Next we introduce some elementary tools in complex harmonic analysis, which will be used frequently in this paper.

For~$z\in\bn$, write~$P_z$~for the orthogonal projection onto the complex line~$\mathbb{C}z$ and~$Q_z=I-P_z$. The function
$$\varphi_z(w)=\frac{z-P_z(w)-(1-|z|^2)^{1/2}Q_z(w)}{1-\langle w,z\rangle}$$
is the (unique) automorphism of~$\bn$~that satisfies~$\varphi_z\circ\varphi_z=id$~and~$\varphi_z(0)=z$. The following Lemma is in Chapter 2 of~\cite{Rudin}.

\begin{lem}\label{basic about varphi}
Suppose~$a$, $z$, $w\in\bn$, then
\begin{itemize}
\item[(1)] $$1-\langle\varphi_a(z),\varphi_a(w)\rangle=\frac{(1-\langle a,a\rangle)(1-\langle z,w\rangle)}{(1-\langle z,a\rangle)(1-\langle a,w\rangle)}.$$
\item[(2)] As a consequence of (1),
$$1-|\varphi_a(z)|^2=\frac{(1-|a|^2)(1-|z|^2)}{|1-\langle z,a\rangle|^2}.$$
\item[(3)] The Jacobian of the automorphism~$\varphi_z$~is
$$(J\varphi_z(w))=\frac{(1-|z|^2)^{n+1}}{|1-\langle w,z\rangle|^{2(n+1)}}.$$
\end{itemize}
\end{lem}

The pseudo-hyperbolic metric on~$\bn$~is defined by
$$\rho(z,w)=|\varphi_z(w)|.$$
And the hyperbolic metric is defined by
$$\beta(z,w)=\frac{1}{2}\log\frac{1+\rho(z,w)}{1-\rho(z,w)}.$$
Thus $\rho(z,w)=\tanh\beta(z,w)$. It's well known that the two metrics are invariant under actions of~$Aut(\bn)$, the group of holomorphic automorphisms of~$\bn$. That is, given~$\psi\in Aut(\bn)$, $$\rho(z,w)=\rho(\psi(z),\psi(w)),$$
$$\beta(z,w)=\beta(\psi(z),\psi(w))$$
for all~$z$, $w\in\bn$. For~$r>0$, $z\in\bn$, write
$$D(z,r)=\{w\in\bn: \beta(w,z)< r\}=\{w\in\bn: \rho(w,z)< s_r\},$$
where~$s_r=\tanh r$. In this paper, we use~$D(z,r)$~to denote the hyperbolic ball in~$\bn$~and use~$D_d(z',r)$~to denote the hyperbolic ball in $\mathbb{B}_d$. The notation~$B(z,\delta)$~is used to denote the Euclidian ball with center~$z$~and radius~$\delta$.

\begin{lem}{\cite[2.2.7]{Rudin}}\label{description of D(z,R)}
For~$z\in\bn$, $r>0$, the hyperbolic ball~$D(z,r)$~consists of all~$w$~that satisfy:
$$\frac{|Pw-c|^2}{s_r^2\rho^2}+\frac{|Qw|^2}{s_r^2\rho}<1,$$
where~$P=P_z$, $Q=Q_z$, and
$$c=\frac{(1-s_r^2)z}{1-s_r^2|z|^2},~~~\rho=\frac{1-|z|^2}{1-s_r^2|z|^2}.$$
\end{lem}
Thus~$D(z,r)$~is an ellipsoid with center~$c$, radius of~$s_r\rho$~in the~$z$~direction and~$s_r\sqrt{\rho}$~in the directions perpendicular to~$z$. Therefore the Lebesgue measure of~$D(z,r)$~is
$$v_n(D(z,r))=Cs_r^{2n}\rho^{n+1},$$
where~$C>0$~is a constant depending only on~$n$.
Note that when we fix~$r$, $\rho$~is comparable with~$1-|z|^2$. Hence~$v(D(z,r))$~is comparable with~$(1-|z|^2)^{n+1}$.

Suppose~$\nu$~is a positive, finite, regular, Borel measure. The operator
$$T_{\nu}f(z)=\int_{\bn}\frac{f(w)}{(1-\langle z,w\rangle)^{n+1}}d\nu(w)$$
defines an analytic function for every~$f\in H^{\infty}$. The following lemma can be seen from the proof of Lemma 2.1 in \cite{Suarez07}.

\begin{lem}\label{4equivalent}
Let~$\nu$~be a positive, finite, regular, Borel measure on~$\bn$~and~$r>0$. Then the following are equivalent.
When one of these conditions holds, $\nu$ is called a Carleson measure ( for $\ber$).
\begin{itemize}
\item[(1)]$\sup_{z\in\bn}\int_{\bn}\frac{(1-|z|^2)^{n+1}}{|1-\langle w, z\rangle|^{2(n+1)}}d\nu(w)<\infty,$
\item[(2)]$\exists C>0:\int|f|^2d\nu\leq C\int|f|^2dv~\mbox{for all}~f\in\ber,$
\item[(3)]$\sup_{z\in\bn}\frac{\nu(D(z,~r))}{v_n(D(z,~r))}<\infty,$
\item[(4)]$T_{\nu}$ extends to a bounded linear operator on~$\ber$.
\end{itemize}
\end{lem}
Suppose~$\nu$~is a Carleson measure, by Fubini's Theorem, we have:
$$\langle Tf,g\rangle=\int_{\bn}f(w)\overline{g(w)}d\nu(w),~~~\forall f, g\in\ber.$$
The operator~$T_{\nu}$~plays an important role in this paper. As a corollary of Theorem 7.3 in~\cite{Suarez07}, we have:
\begin{lem}\label{measure in toeplitz}
Suppose~$\nu$~is a Carleson measure, then the operator~$T_{\nu}$~belongs to the Toeplitz algebra of~$L^{\infty}$~symbols~$\mathcal{T}(L^{\infty})$.
\end{lem}

The following lemma is crucial to our proof of essential normality. One can find a proof in \cite[Proposition 1.4]{Mcdonald}.
\begin{lem}\label{ess-commutant}
If~$f\in C(\clb)$, then~$T_f$~essentially commutes with every operator in the Toeplitz algebra~$\toe$.
\end{lem}

We will use the following lemma frequently in calculation. One can find a proof in \cite[Proposition~1.4.10]{Rudin}.
\begin{lem}\label{1.4.10 Rudin}
For $z\in\bn$, $c$ real, $t>-1$, define
$$
I_c(z)=\int_S\frac{d\sigma(\zeta)}{|1-\langle z,\zeta\rangle|^{n+c}}
$$
and
$$
J_{c,t}(z)=\int_{\bn}\frac{(1-|w|^2)^tdv(w)}{|1-\langle z,w\rangle|^{n+1+t+c}}.
$$
When $c<0$, then $I_c$ and $J_{c,t}$ are bounded in $\bn$. When $c>0$, then
$$
I_c(z)\approx(1-|z|^2)^{-c}\approx J_{c,t}(z).
$$
Finally,
$$
I_0(z)\approx \log\frac{1}{1-|z|^2}\approx J_{0,t}(z).
$$
The notation $a(z)\approx b(z)$ means that the ratio $a(z)/b(z)$ has a positive finite limit as $|z|\to1$.
\end{lem}

\section{Carleson Measure and Holomorphic Extension}\label{mainsec}
In this section we reveal some connection between the Geometric Arveson-Douglas Conjecture and the holomorphic extension theory. Our idea come from Su\'{a}rez's paper \cite{Suarez07}.The following theorem is the main theorem of this section. It's connection with holomorphic extension theory is discussed in Remark \ref{remext} in the end of this section. Also, a new result extending that of \cite{Beatrous} is stated in Remark \ref{remext2} in the next section.
\begin{thm}\label{main}
If there exists a positive, finite, regular, Borel measure~$\mu$~on~$M$~such that the~$L_{\mu}^2$~norm and Bergman norm are equivalent on~$Q$, i.e., $\exists C, c>0$ such that $\forall f\in Q$,
$$c\|f\|^2\leq\int_M|f(w)|^2d\mu(w)\leq C\|f\|^2,$$
then the submodule~$P$~is essentially normal.
\end{thm}

\begin{proof}
First, we prove that the measure~$\mu$~is a Carleson measure. From the assumption, we have for any~$z\in\bn$,
\begin{eqnarray*}
\int_{\bn}\frac{(1-|z|^2)^{n+1}}{|1-\langle w,~z\rangle|^{2(n+1)}}d\mu(w)&=&\int_{M}|k_z(w)|^2d\mu(w)\\
&=&\int_M|Qk_z(w)|^2d\mu(w)\\
&\leq&C\|Qk_z\|^2\\
&\leq&C.
\end{eqnarray*}
The second equality is because $k_z-Qk_z\in P$, therefore $k_z(w)=Qk_z(w), \forall w\in M$.
By Lemma~\ref{4equivalent}, $\mu$ is a Carleson measure.

Next we show that the projection~$P$~is a continuous function calculous of~$T_{\mu}$~and therefore is in the Toeplitz algebra. From the equation
$$\langle T_{\mu}f,~f\rangle=\int_M|f(w)|^2d\mu(w),~~~\forall f\in\ber$$
we see that~$T_{\mu}$~is positive and vanishes on~$P$. Also, the equivalence of~$L^2(\mu)$-norm and Bergman norm on~$Q$~implies that~$T_{\mu}$~is bounded below on~$Q$. Therefore~$0$~is isolated in~$\sigma(T_{\mu})$~and~$P=\ker T_{\mu}$. Take any continuous function~$f$~on~$\mathbb{R}$~that vanishes at~$0$~and equals~$1$~on the rest of the spectrum, then~$Q=f(T_{\mu})$.

Finally, by Lemma~\ref{measure in toeplitz}, $Q$~is in the Toeplitz algebra. By Lemma~\ref{ess-commutant}~and Lemma~\ref{enormal}, the quotient module~$Q$~is essentially normal and so is the submodule $P$. This completes the proof.
\end{proof}

As a consequence of the proof of Theorem~\ref{main}, we have the following:

\begin{cor}\label{toe2}
Assume the same as Theorem~\ref{main}, then the projection operators~$P$~and $Q$ are in the Toeplitz algevra~$\toe$.
\end{cor}

The following example is the starting point of our research.

\begin{exam}\label{hyperplane}
Suppose~$M$~is the intersection of a~$d$-dimension hyperplane and~$\bn$, where~$d<n$. We can identify~$M$~with the unit ball in~$\mathbb{C}^d$. Let~$P$~be the submodule of~$\ber$~consisting of all the functions that vanish on~$M$~and~$Q$ be the orthogonal complement of~$P$.

Let~$\rho$~be the weighted bergman measure on $M$: $d\rho=c(1-|z|^2)^{n-d}dm_d$, where~$m_d$~is the Lebesgue measure on~$M$~and~$c>0$~is chosen such that~$\rho(M)=1$~. It is well-known that the weighted Bergman space on~$M$~defined by
$$L_{a,n-d}^2(M)=\{f~\mbox{is an analytic function on}~M:~\int_M|f(w)|^2d\rho(w)<\infty\}$$
is a reproducing kernel Hilbert space on $M$ with reproducing kernels
$$\{K_z^{n-d}(w)=\frac{1}{(1-\langle w,~z\rangle)^{n+1}}:~z\in M\}.$$
Since~$Q$~is also the reproducing kernel Hilbert space with the same reproducing kernels, one can identify the two spaces(cf.\cite{Aronszajn}).
It's also clear that~$\rho$~is a Carleson measure for the~$n$-dimensional Bergman space~$\ber$. Then the equation
$$\langle T_{\rho}f,~g\rangle=\int_Mf(w)\overline{g(w)}d\rho(w),~~~\forall f,~g\in\ber$$
shows that~$T_{\rho}=Q$.
\end{exam}

\begin{rem}
 From Example~\ref{hyperplane}, one might suspect that there is always a measure on $M$ that defines the same norm on~$Q$. We show that this is not the case, even when~$M$~consists of finite points.

\begin{exam}
Suppose~$M=\{a_1,\cdots,a_m\}\subset\bn$, then there exists a positive measure~$\mu$~on~$M$~such that~$\forall f\in Q$,
    $$\|f\|^2=\int_M |f|^2d\mu$$
if and only if~$m=1$.

\begin{proof}
The ``if'' part is obvious, we prove the ``only if'' part.

Suppose~$\mu$~is supported on~$M$~such that the equation holds, then~$\mu=\sum\limits_{i=1}^mc_i\delta_i$, where~$c_i\geq0$~and~$\delta_i$~are the point masses at~$a_i$, $i=1,\cdots,m$. For any tuple~$x_1,\cdots,x_m\in\mathbb{C}$, let~$x=\sum\limits_{i=1}^mx_ik_{a_i}\in Q$. Then
$$
\|x\|^2=\sum_{i,j}x_i\overline{x_j}\langle k_{a_i}, k_{a_j}\rangle.
$$

On the other hand,
\begin{eqnarray*}
\int_M|x(w)|^2d\mu(w)&=&\sum_{i=1}^m c_i|x(a_i)|^2\\
&=&\sum_{i=1}^mc_i(1-|a_i|^2)^{-(n+1)}|\sum_{j=1}^mx_j\langle k_{a_j}, k_{a_i}\rangle|^2.
\end{eqnarray*}
Let $G$ be the $m\times m$ matrix $(\langle k_{a_i}, k_{a_j}\rangle)_{ij}$, then
$$
\|x\|^2=\begin{pmatrix}x_1& \dots & x_m\end{pmatrix}G\begin{pmatrix}\overline{x_1}\\ \vdots\\ \overline{x_m}\end{pmatrix}
$$
and
$$
\int_M|x(w)|^2d\mu(w)=\begin{pmatrix}x_1&\dots&x_m\end{pmatrix}G\begin{pmatrix}d_1&\dots&0\\ \vdots&\ddots&\vdots\\0&\dots&d_m\end{pmatrix}G^*\begin{pmatrix}\overline{x_1}\\ \vdots\\ \overline{x_m}\end{pmatrix}
$$
where $d_i=c_i(1-|a_i|^2)^{-(n+1)}$. Since $x_i$ are arbitrary, we have
$$
G=G\begin{pmatrix}d_1&\dots&0\\ \vdots&\ddots&\vdots\\0&\dots&d_m\end{pmatrix}G^*.
$$
This only holds when $G$ is diagonal, which implies $m=1$.
\end{proof}
\end{exam}
\end{rem}

\begin{rem}\label{remext}
Given a positive Carleson measure $\mu$ on $M$, the restriction map
$$
R: \ber\to L^2(\mu), f\mapsto f|_M
$$
is bounded. Assume that $\ker R=P$, then by the open mapping theorem, the hypotheses in Theorem \ref{main} is equivalent to that $R$ has closed range. If $L$ is a closed subspace of $L^2(\mu)$ containing $Range R$ and there is a bounded linear operator
$E: L\to \ber$ such that $RE=Id_{L}$, then $Range R=L$, therefore is closed. These seems redundant but will be very useful in practice. For example, given Theorem 4.3 in \cite{DYT}, one easily sees that the weighted measure in \cite{DYT} satisfies Theorem \ref{main}. So the essential normality of the corresponding submodule follows. On the other hand, suppose $\mu$ satisfies the assumption of Theorem \ref{main}, the map
$$
E: Range R\to Q\subset\ber, Rf\mapsto f\in Q
$$
is an extension operator. So this also provides a way to solve the  holomorphic extension problem related to the topic in \cite{Beatrous}.
\end{rem}

\begin{rem}
Theorem \ref{main} also relates to the reverse Carleson inequality(cf. \cite{Luecking}). One can think of the requirements in Theorem \ref{main} as a reverse Carleson inequality on the zero variety.
\end{rem}

\begin{cor}
Suppose~$M$~is an interpolating sequence, then the corresponding projection~$P$~is in the Toeplitz algebra and the submodule~$P$~is essentially normal.
\end{cor}

\section{Geometric Arveson-Douglas Conjecture}\label{manifoldsec}

In this section, we construct a measure satisfying the hypotheses of Theorem \ref{main} for any complex analytic subset that intersects $\pbn$ transversally and has no singular point on $\pbn$.
\begin{defn}
Let $\Omega$ be a complex manifold. A set $A\subset\Omega$ is called a \emph{(complex) analytic subset} of $\Omega$ if for each point $a\in\Omega$ there are a neighborhood $U\ni a$ and functions $f_1,\cdots,f_N$ holomorphic in this neighborhood such that
$$
A\cap U=\{z\in U: f_1(z)=\cdots=f_N(z)=0\}.
$$
A point $a\in A$ is called \emph{regular} if there is a neighborhood $U\ni a$ in $\Omega$ such that $A\cap U$ is a complex submanifold of $\Omega$. A point $a\in A$ is called a \emph{singular point} of $A$ if it's not regular.
\end{defn}
\begin{defn}
Let $Y$ be a manifold and let $X, Z$ be two submanifolds of $Y$. We say that the submanifolds $X$ and $Z$ are \emph{transversal} if $\forall x\in X\cap Z$, $T_x(X)+T_x(Z)=T_x(Y)$.
\end{defn}

\begin{thm}\label{manifold}
Suppose~$\tilde{M}$~is a complex analytic subset of an open neighborhood of~$\clb$~satisfying the following conditions:
\begin{itemize}
\item[(1)] $\tilde{M}$~intersects~$\pbn$~transversally.
\item[(2)] $\tilde{M}$~has no singular point on~$\pbn$.
\end{itemize}
Let~$M=\tilde{M}\cap\bn$ and let~$P=\{f\in\ber: f|_M=0\}$. Then the submodule~$P$~is essentially normal.
\end{thm}
Note that in this case, condition (1) is equivalent to that $\tilde{M}$ is not tangent with $\pbn$ at every point of $\tilde{M}\cap\pbn$. Condition (2) implies that $\tilde{M}$ has only finite singular points inside $\bn$.

\begin{cor}
Under the assumption of Theorem \ref{manifold}, the projection operator~$P$~is in the Toeplitz algebra $\toe$.
\end{cor}
In order to prove Theorem~\ref{manifold}, we need to establish a few lemmas.

\begin{lem}\label{integral on affine space}
Let~$\alpha$~be the intersection of a~$d$-dimensional affine space and~$\bn$. Then $\alpha$ is a $d$-dimensional ball. Let~$r$~be the radius of~$\alpha$~and~$v$~be the volume measure on~$\alpha$. Then for any function~$f$~holomorphic on~$\alpha$~and any~$R>0$, $z\in\alpha$,
$$\int_{\alpha\cap D(z,R)}f(w)\frac{(1-|w|^2)^{n-d}}{(1-\langle z,w\rangle)^{n+1}}dv(w)=r^{-2}C_Rf(z).$$
where
$$C_R=\int_{D_d(0,R)}(1-|w|^2)^{n-d}d\nu(w).$$
Here~$D_d(0,R)$~means the hyperbolic ball in~$\mathbb{B}_d$~centered at~$0$~with radius~$R$~and~$\nu$~is the volume measure on~$\mathbb{B}_d$.
\end{lem}

\begin{proof}
Let~$z_0$~be the center of $\alpha$~and let
$$\phi:~\alpha\to\beta=\frac{1}{r}(\alpha-z_0),~z\mapsto\frac{1}{r}(z-z_0).$$
The affine space~$\beta$ is the intersection of a hyperplane and~$\bn$, therefore can be identified with~$\mathbb{B}_d$. Clearly~$\phi$~is biholomorphic. For~$z\in\alpha$, consider the map
$$\varphi_z\phi^{-1}:~\beta\to\gamma=\varphi_z(\alpha).$$
By~\cite[Proposition 2.4.2]{Rudin}, $\gamma$~is an affine space containing~$0$. Hence~$\gamma$~can also be identified with~$\mathbb{B}_d$. So~$\varphi_z\phi^{-1}$~is an automorphism of~$\mathbb{B}_d$~and therefore preserves the hyperbolic metric. We get
\begin{eqnarray*}
\phi(D(z,R)\cap\alpha)
&=&\phi\varphi_z^{-1}\varphi_z(D(z,R)\cap\alpha)=\phi\varphi_z^{-1}(D(0,R)\cap\gamma)\\
&=&\phi\varphi_z^{-1}(D_d(0,R))=D_d(\phi(z),R).
\end{eqnarray*}
Therefore
\begin{eqnarray*}
&&\int_{\alpha\cap D(z,R)}f(w)\frac{(1-|w|^2)^{n-d}}{(1-\langle z,w\rangle)^{n+1}}dv(w)\\
&=&\int_{D_d(\phi(z),R)}f\phi^{-1}(\eta)\frac{(1-|\phi^{-1}(\eta)|^2)^{n-d}}{(1-\langle z,\phi^{-1}(\eta)\rangle)^{n+1}}d\nu(\phi^{-1}(\eta))\\
&=&\int_{D_d(\phi(z),R)}f\phi^{-1}(\eta)\frac{(r^2-r^2|\eta|^2)^{n-d}}{(r^2-r^2\langle\phi(z),\eta\rangle)^{n+1}}r^{2d}d\nu(\eta)\\
&=&r^{-2}\int_{D_d(\phi(z),R)}f\phi^{-1}(\eta)\frac{(1-|\eta|^2)^{n-d}}{(1-\langle\phi(z),\eta\rangle)^{n+1}}d\nu(\eta)\\
&=&r^{-2}C_Rf(z).
\end{eqnarray*}
The last equation comes from the following argument.

In general, if~$g$~is holomorphic on~$\mathbb{B}_d$, for~$R>0$~and~$\xi\in\mathbb{B}_d$,
\begin{eqnarray*}
&&\int_{D_d(\xi,R)}g(\eta)\frac{(1-|\eta|^2)^{n-d}}{(1-\langle\xi,\eta\rangle)^{n+1}}d\nu(\eta)\\
&=&\int_{D_d(0,R)}g\varphi_{\xi}(w)\frac{(1-|\varphi_{\xi}(w)|^2)^{n-d}}{|1-\langle\xi,\varphi_{\xi}(w)\rangle)^{n+1}}\frac{(1-|\xi|^2)^{d+1}}{(1-\langle w,\xi\rangle|^{2(d+1)}}d\nu(w)\\
&=&\int_{D_d(0,R)}g\varphi_{\xi}(w)\frac{(1-\langle\xi,w\rangle)^{n+1}(1-|\xi|^2)^{n-d}(1-|w|^2)^{n-d}(1-|\xi|^2)^{d+1}}{(1-|\xi|^2)^{n+1}|1-\langle\xi,w\rangle|^{2(n-d)}|1-\langle w,\xi\rangle|^{2(d+1)}}d\nu(w)\\
&=&\int_{D_d(0,R)}g\varphi_{\xi}(w)\frac{(1-|w|^2)^{n-d}}{(1-\langle w,\xi\rangle)^{n+1}}d\nu(w)\\
&=&C_Rg(\xi).
\end{eqnarray*}
This completes the proof.
\end{proof}

\begin{lem}\label{integral on Bd}
For~$t>0$, we have
$$\lim_{r\to1-}\sup_{z\in\mathbb{B}_d}\int_{w\in\mathbb{B}_d:r<|w|<1}\frac{(1-|w|^2)^t}{|1-\langle z,w\rangle|^{d+1}}d\nu(w)=0.$$
\end{lem}

\begin{proof}
Let
$$I(z)=\int_S\frac{1}{|1-\langle z,\zeta\rangle|^{d+1}}d\sigma(\zeta).$$
Where~$S$~is the unit sphere in~$\mathbb{C}^d$~and~$\sigma$~is the volume measure on~$S$.
By Lemma \ref{1.4.10 Rudin}, there exists~$C>0$~such that
$$I(z)\leq C(1-|z|^2)^{-1}.$$
Hence
\begin{eqnarray*}
\int_{r<|w|<1}\frac{(1-|w|^2)^t}{|1-\langle z,w\rangle|^{d+1}}d\nu(w)&=&\int_r^1\int_S\frac{(1-s^2)^t}{|1-\langle z,s\zeta\rangle|^{d+1}}s^{2d-1}d\sigma(\zeta)ds\\
&=&\int_r^1(1-s^2)^ts^{2d-1}I(sz)ds\\
&\leq&C\int_r^1(1-s^2)^t(1-|sz|^2)^{-1}ds\\
&\leq&C\int_r^1(1-s^2)^{t-1}ds\to 0.~~~~~~(r\to1-)
\end{eqnarray*}
This completes the proof.
\end{proof}

Suppose~$\tilde{M}$~is as in Theorem~\ref{manifold}. We first assume that $\tilde{M}$ is connected. For~$0\leq s<t\leq 1$, define
$$M_s^t=\{z\in M|~s\leq|z|<t\}.$$
Write $M_s=M_s^1, M^t=M_0^t$.
Since~$\tilde{M}$~has no singular point on~$\pbn$, we can cover~$\pbn\cap\tilde{M}$~with finite open sets~$\{U_i\}$, $U_i\subset\tilde{M}$~such that:
\begin{itemize}
\item[(1)] For each~$i$, we can find~$n-1$~of the canonical basis of~$\mathbb{C}^n$, denoted $e_{i_1},\cdots,e_{i_{n-1}}$~such that for any~$z\in U_i$, the~$n$~vectors~$\{z, e_{i_1},\cdots, e_{i_{n-1}}\}$~spans~$\mathbb{C}^n$.
\item[(2)] $\tilde{M}$~has local coordinates on each~$U_i$, i.e., there exists open set~$\Omega_i\subset\mathbb{C}^d$~and ~$\varphi_i:\Omega_i\to U_i$~which is one to one and holomorphic.
\end{itemize}
Fix~$z\in U_i$, apply the Gram-Schimidt process to~$\{z,e_{i_1},\cdots,e_{i_{n-1}}\}$~to obtain a new basis~$\{f_1^z,\cdots,f_n^z\}$, then~$z=(z_1,0,\cdots,0)$ under this basis. Let~$G^z:\Omega_i\to U_i$, $G^z=(g_1^z,\cdots,g_n^z)$~be the expression of~$\varphi_i$~under the new basis. Note that the new basis and expression depend continuously on~$z$.

Since~$\tilde{M}$~intersects~$\pbn$~transversally, by possibly refining the cover~$\{U_i\}$~we can assume that for each~$U_i$, $\forall z\in U_i$, $(\frac{\partial g_1^z}{\partial z_1},\cdots,\frac{\partial g^z_1}{\partial z_d})$ is non-zero at~$z$. Since the matrix~$[\frac{\partial g_i^z}{\partial z_j}(z)]_{1\leq i\leq n,1\leq j\leq d}$ has rank~$d$, by possibly refining~$\{U_i\}$~again we could get~$2\leq k_1,\cdots,k_{d-1}\leq n$~for each~$U_i$, such that the determinant~$\frac{\partial(g_1^z,g_{k_1}^z,\cdots,g_{k_{d-1}}^z)}{\partial(z_1,\cdots,z_d)}|_z\neq0$, $\forall z\in U_i$. Let~$\epsilon$~be the Lebesgue number of the cover~$\{U_i\}$~and let~$V_i=\{z\in U_i|~d(z,\partial U_i)>\frac{1}{2}\epsilon\}$, then~$\pbn\cap\tilde{M}\subset\cup V_i$. The function~$\frac{\partial(g_1^z,g_{k_1}^z,\cdots,g_{k_{d-1}}^z)}{\partial(z_1,\cdots,z_d)}(w)$~is uniformly continuous on~$\{(z,w)|z\in\bar{V_i},w\in U_i\}$. Therefore~$\exists\delta>0$~such that~$\forall z\in V_i$, $\forall w\in B(z,\delta)$, $\frac{\partial(g_1^z,g_{k_1}^z,\cdots,g_{k_{d-1}}^z)}{\partial(z_1,\cdots,z_d)}(w)\neq0$. By the implicit function theorem, we have:

\begin{lem}
There exists a finite open cover~$\{V_i\}$~of~$\pbn\cap\tilde{M}$~and~$\delta>0$~such that for any fixed~$V_i$, we can pick~$d-1$~numbers out of~$\{2,\cdots,n\}$, assume they are~$\{2,\cdots,d\}$~without loss of generality, such that~$\forall z\in\bar{V_i}$~and~$\forall w\in B(z, \delta)$,
$$w=(w_1,\cdots,w_d,F_{d+1}^z(w'),\cdots,F_n^z(w'))$$
under the basis~$\{f_1^z,\cdots,f_n^z\}$, where~$w'=(w_1,\cdots,w_d)$. The functions~$F_i^z(w')$~are holomorphic on~$w'$~and depends continuously on~$z$.
\end{lem}

In the later discussion, whenever we fix a~$z\in V_i$, we will discuss under the new basis~$\{f_i^z\}_{i=1}^n$~and the new expression~$(w',F_{d+1}^z,\cdots,F_n^z)$~and we will omit the superscript ``$z$'' for convenience. Moreover, we will denote any constant that depends only on~$M$~by~$C$~as long as it doesn't cause confusion. So~$C$~may refer to different constant in different places.

By Proposition 1 in \cite[Page 31]{Complex analytic sets}, the assumptions in Theorem \ref{manifold} implies $\tilde{M}$ has only finite singular points in $\bn$. Let~$\Sigma=\{z_1,\cdots,z_m\}$~be the set of all singular points of~$\tilde{M}$~inside~$\bn$.
Take~$0<s_1<1$~such that~$\Sigma\cap M_{s_1}=\emptyset$. Then the volume measure~$v_d$~is well-defined on~$M_{s_1}$. In local coordinates, $v_d$~corresponds to the volume form~$E(w)dx_1\wedge dy_1\wedge\cdots\wedge dy_d$, where~$E(w)$~is the square root of the absolute value of the determinant of the matrix representation of the metric tensor on~$M_{s_1}$. Note that~$E(w)$ is uniformly continuous on~$z$~and~$w$. Let
$$\delta=\sum_{i=1}^m(1-|z_i|^2)^{n+1}\delta_{z_i},$$ where~$\delta_{z_i}$~is the point mass at~$z_i$. For~$s_1<s<1$, let
$$d\mu_s=(1-|w|^2)^{n-d}dv_d|_{M_s}+d\delta.$$

We will prove that for~$s$~sufficiently close to~$1$, $\mu_s$~satisfies the assumption of Theorem~\ref{main}, therefore Theorem~\ref{manifold} holds.

Fix~$z\in V_i$(and the basis depending on~$z$), define a map
$$p_z:\tilde{M}\cap B(z,\delta)\to T\tilde{M}|_z$$
$$(w',F_{d+1}(w'),\cdots,F_n(w'))\mapsto (w',\sum_{i=1}^d\frac{\partial F_{d+1}}{\partial w_i}(z')(w_i-z_i),\cdots,\sum_{i=1}^d\frac{\partial F_n}{\partial w_i}(z')(w_i-z_i))$$
Here $T\tilde{M}|_z$ is the tangent space of $\tilde{M}$ at $z$. Note that by construction, $F_i(z')=0, i=d+1,\cdots,n$.
Clearly, $p_z$~is one to one and holomorphic, $p_z(w)-w\perp z$~and
$$|p_z(w)-w|=O(|w'-z'|^2).$$

\begin{lem}\label{s2}
Fix~$R>0$, then there exists~$1>s_2>s_1$, such that
\begin{itemize}
\item[(1)] $\forall z\in M_{s_2}$, $D(z,R)\subset B(z,\delta)$.
\item[(2)] $\forall z\in M_{s_2}$, $\forall w\in D(z,R)$, $p_z(w)\in\bn$.
\item[(3)] $\sup\limits_{w\in D(z,R)}|\frac{1-|p_z(w)|^2}{1-|w|^2}-1|\to0,~~~|z|\to1$.
\item[(4)] $\sup\limits_{w\in D(z,R)}\beta(p_z(w),w)\to 0$,~~~$|z|\to1$.
\end{itemize}
\end{lem}

\begin{proof}
By Lemma~\ref{description of D(z,R)}, it's easy to see that (1) holds as long as we take~$s_2$~sufficiently close to~$1$.

To prove (2), we notice first that Lemma~\ref{description of D(z,R)} also implies
$$\sup_{w\in D(z,R)}|w-z|=O((1-|z|^2)^{\frac{1}{2}}).$$
Therefore
$$\sup_{w\in D(z,R)}|p_z(w)-w|=O(1-|z|^2).$$
Since~$\langle z,p_z(w)\rangle=\langle z,w\rangle\neq0$, $\varphi_z(p_z(w))$~is well defined. It's easy to verify that
$$\varphi_z(\xi)\in\bn \mbox{ if and only if }\xi\in\bn.$$
So we only need to make sure that~$\varphi_z(p_z(w))\in\bn$. Since
$$|\varphi_z(p_z(w))-\varphi_z(w)|=\frac{(1-|z|^2)^{\frac{1}{2}}}{|1-\langle w,z\rangle|}O(1-|z|^2)=O((1-|z|^2)^{\frac{1}{2}})$$
and
$$|\varphi_z(w)|\leq s_{\sss R},$$
when we take~$s_2$~sufficiently close to~$1$, we have~$|\varphi_z(p_z(w))|<1$. Therefore (2) is proved.

We prove (4) first. Take~$s_2$~so close to~$1$~that~$\forall z\in M_{s_2}$, $\forall w\in D(z,R)$, $\varphi_z(p_z(w))\in D(0,2R)$.
On~$D(0,2R)$, the hyperbolic distance and Euclidian distance are equivalent. Hence
$$\beta(p_z(w),w)=\beta(\varphi_z(p_z(w)),\varphi_z(w))\leq C|\varphi_z(p_z(w))-\varphi_z(w)|\to0,$$
as~$|z|\to1$.

Finally, since~$p_z(w)=\varphi_w\varphi_w(p_z(w))$~and~$|\varphi_w(p_z(w))|\to0$,
apply Lemma~\ref{basic about varphi} (2), we have
$$\frac{1-|p_z(w)|^2}{1-|w|^2}=\frac{1-|\varphi_w((p_z(w)))|^2}{|1-\langle\varphi_w(p_z(w)),w\rangle|^2}.$$
Notice that $|\varphi_w(p_z(w))|=\rho(w,p_z(w))$ tends to $0$ uniformly. We have (3). This completes the proof.
\end{proof}

\begin{lem}
For~$1>s>s_1$, the measure
$$d\mu_{s}=(1-|w|^2)^{n-d}dv_d|_{M_{s}}+\sum_{i=1}^m(1-|z_i|^2)^{n+1}\delta_{z_i}$$
is a Carleson measure.
\end{lem}

\begin{proof}
Fix~$R>0$, by Lemma~\ref{4equivalent}, we only need to prove that
$$\int_{D(z,R)\cap M_s}(1-|w|^2)^{n-d}dv_d(w)\leq C(1-|z|^2)^{n+1}$$
for some constant~$C>0$. Since
$$\frac{1-|w|^2}{1-|z|^2}=\frac{1-|\varphi_z(w)|^2}{|1-\langle\varphi_z(w),z\rangle|^2}\leq C,$$
it suffices to show
$$v_d(D(z,R))\leq C(1-|z|^2)^{d+1}.$$

Since
$$v_d(D(z,R))=\int_{\{w':w\in D(z,R)\}}E(w')dv(w')\leq C\int_{\{w':w\in D(z,R)\}}dv(w').$$
By definition, $|\varphi_{z'}(w')|\leq|\varphi_z(w)|$, so
$$\{w':w\in D(z,R)\}\subset D_d(z',R).$$
Therefore
$$v_d(D(z,R))\leq C(1-|z|^2)^{d+1}.$$

This completes the proof.
\end{proof}

\begin{lem}\label{delta}
There exists a constant~$C>0$~such that
$$T_{\delta}^3>CT_{\delta}.$$
\end{lem}

\begin{proof}
The lemma follows from the fact that~$T_{\delta}$~has closed range and is positive.
\end{proof}

\begin{lem}\label{s3}
For any~$\epsilon>0$, there exists~$1>s_3>s_1$~and~$R>0$~such that,
\begin{itemize}
\item[(1)]
$$\sup_{z\in M_{s_3}}\int_{M_{s_3}}\frac{(1-|z|^2)^{\frac{n-d}{2}}(1-|w|^2)^{\frac{n-d}{2}}}{|1-\langle z,w\rangle|^{n+1}}dv_d(w)<\infty.$$
\item[(2)] $\forall z\in M_{s_3}$,
$$\int_{M_{s_3}\backslash D(z,R)}\frac{(1-|z|^2)^{\frac{n-d}{2}}(1-|w|^2)^{\frac{n-d}{2}}}{|1-\langle z,w\rangle|^{n+1}}dv_d(w)<\epsilon.$$
\end{itemize}
\end{lem}

\begin{proof}
We prove (1) and (2) together. For~$z\in\bar{V_i}\cap M_{s_1}$~and~$R>0$,
\begin{eqnarray*}
& &\int_{M_{s_3}}\frac{(1-|z|^2)^{\frac{n-d}{2}}(1-|w|^2)^{\frac{n-d}{2}}}{|1-\langle z,w\rangle|^{n+1}}dv_d(w)\\
&\leq&\int_{B(z,\delta)\cap M_{s_3}}\frac{(1-|z|^2)^{\frac{n-d}{2}}(1-|w|^2)^{\frac{n-d}{2}}}{|1-\langle z,w\rangle|^{n+1}}dv_d(w)\\
&+&\int_{M_{s_3}\backslash B(z,\delta)}\frac{(1-|z|^2)^{\frac{n-d}{2}}(1-|w|^2)^{\frac{n-d}{2}}}{|1-\langle z,w\rangle|^{n+1}}dv_d(w)
\end{eqnarray*}
For the second part, the integrand is smaller than~$C\delta^{-2(n+1)}(1-s_3^2)^{n-d}$~because
$$|1-\langle z,w\rangle|\geq(1-Re\langle z,w\rangle)\geq\frac{1}{2}(|z|^2+|w|^2-2Re\langle z,w\rangle)=\frac{1}{2}|z-w|^2.$$
So when~$s_3$~is close to~$1$, the second part will be smaller than~$\frac{1}{2}\epsilon$.

For the first part,
\begin{eqnarray*}
&&\int_{B(z,\delta)\cap M_{s_3}}\frac{(1-|z|^2)^{\frac{n-d}{2}}(1-|w|^2)^{\frac{n-d}{2}}}{|1-\langle z,w\rangle|^{n+1}}dv_d(w)\\
&=&\int_{\{w':w\in B(z,\delta)\cap M_{s_3}\}}\frac{(1-|z'|^2)^{\frac{n-d}{2}}(1-|w|^2)^{\frac{n-d}{2}}}{|1-\langle z',w'\rangle|^{n+1}}E(w')dv(w')\\
&\leq&C\int_{\mathbb{B}_d}\frac{(1-|z'|^2)^{\frac{n-d}{2}}(1-|w'|^2)^{\frac{n-d}{2}}}{|1-\langle z',w'\rangle|^{n+1}}dv(w')\\
&=&C\int_{\mathbb{B}_d}\frac{(1-|z'|^2)^{\frac{n-d}{2}}(1-|\varphi_{z'}(\eta')|^2)^{\frac{n-d}{2}}}{|1-\langle z',\varphi_{z'}(\eta')\rangle|^{n+1}}\frac{(1-|z'|^2)^{d+1}}{|1-\langle z',\eta'\rangle|^{2(d+1)}}dv(\eta')\\
&=&C\int_{\mathbb{B}_d}\frac{(1-|\eta'|^2)^{\frac{n-d}{2}}}{|1-\langle z',\eta'\rangle|^{d+1}}dv(\eta').
\end{eqnarray*}
Where the second equality from the bottom is by change of variable~$w'=\varphi_{z'}(\eta')$. By the proof of Lemma~\ref{integral on Bd}, the integral above is uniformly bounded, this proves (1).

The above argument also gives
\begin{eqnarray*}
&&\int_{M_{s_3}\cap B(z,\delta)\backslash D(z,R)}\frac{(1-|z|^2)^{\frac{n-d}{2}}(1-|w|^2)^{\frac{n-d}{2}}}{|1-\langle z,w\rangle|^{n+1}}dv_d(w)\\
&\leq&\int_{\{\varphi_{z'}(w'):w\in B(z,\delta)\cap M_{s_3}\backslash D(z,R)\}}\frac{(1-|\eta'|^2)^{\frac{n-d}{2}}}{|1-\langle z',\eta'\rangle|^{d+1}}dv(\eta').
\end{eqnarray*}
Claim: There exists~$c>0$~such that for any~$R>0$,
$$\{\varphi_{z'}(w'):w\in B(z,\delta)\cap M_{s_3}\backslash D(z,R)\}\cap cs_R\mathbb{B}_d=\emptyset.$$
Assume the claim, then (2) follows from Lemma~\ref{integral on Bd}.

Now we prove the claim. For~$z\in M_{s_1}$, $w\in B(z,\delta)$, let~$\eta=\varphi_z(w)$, $\eta'$~be the first~$d$~entries of~$\eta$. Then~$\eta'=\varphi_{z'}(w')$.
\begin{eqnarray*}
|\eta|^2-|\eta'|^2&=&\frac{1-|z|^2}{|1-\langle w,z\rangle|^2}\sum_{i=d+1}^n|f_i^z(w')|^2\\
&\leq& C\frac{1-|z'|^2}{|1-\langle w',z'\rangle|^2}|w'-z'|^2\\
&\leq& C\bigg(\frac{1}{|1-\langle w',z'\rangle|^2}|z_1-w_1|^2+\sum_{i=2}^d\frac{1-|z'|^2}{|1-\langle w',z'\rangle|^2}|w_i|^2\bigg)\\
&=&C|\eta'|^2.
\end{eqnarray*}
Thus
$$|\eta|^2\leq(C+1)|\eta'|^2.$$
If~$w\notin D(z,R)$, then~$|\eta|=|\varphi_z(w)|\geq s_{\sss R}$. Therefore~$|\eta'|\geq\frac{1}{\sqrt{C+1}}s_{\sss R}$. Take~$c=\frac{1}{\sqrt{C+1}}$~and the proof is complete.
\end{proof}

\begin{proof}{\textbf{proof of Theorem~\ref{manifold}}}
First, we prove the theorem under the assumption that $M$ is connected. Then the dimension of $M$ at every regular point is the same. Let $0<d<n$ be the dimension.

Let~$\epsilon>0$~be determined later. Let~$R>0$~and~$s_3$, $s_2$~be  as in Lemma~\ref{s3}~and Lemma~\ref{s2}. Let~$s=\max\{s_2, s_3\}$~and~$1>s'>s$~be such that~$\forall z\in M_{s'}$, $D(z,2R)\cap M\subset M_s$. We may enlarge~$s$~(and the associated~$s'$~)in the proof and still denote it by~$s$~(~$s'$ ).

We will prove that~$T_{\mu_s}^3\geq cT_{\mu_s}$~for some~$c>0$. Since~$T_{\mu_s}$~is self-adjoint and~$\ker T_{\mu_s}=P$. $T_{\mu_s}$ is bounded below on~$Q$. This will give us the desired result, by Theorem~\ref{main}.

Denote $\mathcal{P}(\mu_s)$ to be the closure of the restriction of all analytic polynomials to $M$ in~$L^2(\mu_s)$~. Clearly $Range R\subset\mathcal{P}(\mu_s)$. Suppose $s'<t<1$, for every $z\in M_{s'}^t$, there is an open neighborhood $U\ni z$ contained in $M_s$ and doesn't touch $\pbn$ such that $M$ has local coordinates on $U$. It's easy to prove that for a compact set $V\subset U$, there is a constant $C>0$ such that $\forall p\in \mathbb{C}[z_1,\cdots,z_n]$, $\forall z\in V$,
$$|p(z)|^2\leq C\int_U|p(w)|^2dv_d(w)\leq C'\int_M|p(w)|^2d\mu_s(w).$$
Clearly the same is true for $z=z_i, i=1,\cdots,m$.
Suppose $K\subset M$ is compact, then $K$ is contained in $M^t$ for some $t<1$. Assume $t>s'$, we can cover $\overline{M_{s'}^t}$ with finite compact neighborhoods $V_i$ as above. So there is a constant $C>0$ such that for any analytic polynomial $p$, $\forall z\in M_{s'}^t\cup \Sigma$,

$$
|p(z)|^2\leq C\int_M|p(w)|^2d\mu_s(w).
$$
For $z\in M^{s'}\backslash\Sigma$, using the maximum modulus principle, we have
$$
|p(z)|^2\leq\sup_{w\in M_{s'}^t\cup\Sigma}|p(w)|^2\leq C\int_M|p(w)|^2d\mu_s(w).
$$
This means the evaluation at every point in $M$ is bounded on $\mathcal{P}(\mu_s)$. Therefore we can think of $f\in\mathcal{P}(\mu_s)$ as a pointwisely defined function on $M$(instead of an equivalence class in $L^2(\mu_s)$). Also, it's easy to prove that under this definition, $\forall f\in\ber, \forall z\in M, Rf(z)=f(z)$.

In conclusion, the space~$\mathcal{P}(\mu_s)$~is a reproducing kernel Hilbert space on~$M$, and the reproducing kernels on any compact subset are uniformly bounded.

Consider the operator
$$T:\mathcal{P}(\mu_s)\to L^2(\mu_s),~Tf=f\chi_{\sss M_s^{s'}}.$$
Then~$T$~is compact: suppose $\{f_k\}\subset \mathcal{P}(\mu_s)$ and $f_k$ weakly converges to $0$. Then $f_k$ converges to $0$ pointwisely and are uniformly bounded on $M_s^{s'}$. By the above argument and the dominance convergence theorem,
$$
\|Tf_k\|^2=\int_{M_s^{s'}}|f(w)|^2d\mu_s(w)\to0.
$$
So $T$ is compact, therefore $|T|$ is compact.
Since $\|Tf\|=\||T|f\|$, $\forall f\in\mathcal{P}(\mu_s)$. Using the spectral decomposition of $|T|$, we see that for any~$0<a<1$, there exists a finite codimensional subspace~$L\subset \mathcal{P}(\mu_s)$, such that~$\forall f\in L$,
$$\int_{M_s^{s'}}|f|^2d\mu_s\leq(1-a)\int_M|f|^2d\mu_s,$$
so
$$\int_M|f|^2d\mu_{s'}\geq a\int_M|f|^2d\mu_s.$$
We will use this in the last part of our proof.

Define the operator
$$\tTmus: \Ps\to\Ps$$
$$\tTmus f(z)=\int_Mf(w)\frac{1}{(1-\langle z,w\rangle)^{n+1}}d\mu_s(w),~\forall z\in M$$
By definition, $\forall F\in\ber$, $\tTmus Rf=R\Tmus F$. Since
$$
\|R\Tmus F\|_{\mu_s}^2=\langle \Tmus^3 F, F\rangle\leq\|\Tmus\|^2\langle\Tmus F, F\rangle=\|RF\|_{\mu_s}^2
$$
$\tTmus$ is bounded on $\Ps$. We will show that $\tTmus$ is bounded below.

For~$z\in M_{s'}$, $f=RF\in\Ps$, $F\in\ber$,
$$\tTmus f(z)=\int_{\Sigma}f(w)K_w(z)d\mu_s(w)+\int_{M_s}f(w)\frac{(1-|w|^2)^{n-d}}{(1-\langle z,w\rangle)^{n+1}}dv_d(w).$$
Consider the map~$p_z: D(z,2R)\cap M\to TM|_z$~defined before Lemma~\ref{s2}, by (4) of Lemma~\ref{s2}, by enlarging~$s$, we could assume~$\beta(p_z(w),w)<\frac{1}{2}R$, $\forall w\in D(z,2R)$.
Therefore
$$p_z(D(z,2R)\cap M)\supset D(z,\frac{3}{2}R)\cap TM|_z$$
and
$$p_z^{-1}(D(z,\frac{3}{2}R)\cap TM|_z)\supset D(z,R)\cap M.$$
Define

$$
I(z)=
\begin{cases}
\int_{\Sigma}f(w)K_w(z)d\mu_s(w) & z\in\Sigma\\
\int_{p_z^{-1}(D(z,\frac{3}{2}R)\cap TM|_z)}f(w)\frac{(1-|w|^2)^{n-d}}{(1-\langle z,w\rangle)^{n+1}}dv_d(w) & z\in M_{s'}
\end{cases}
$$
and
$$
II(z)=\tTmus f(z)-I(z)
$$
Then~$I(z)+II(z)=\tTmus f(z)$, $\forall z\in M_{s'}\cup\Sigma$.

For~$I(z)$,
$$\int_{M}|I(z)|^2d\mu_{s'}=\langle T_{\delta}^3F,F\rangle+\int_{M_{s'}}|I(z)|^2(1-|z|^2)^{n-d}dv_d(z).$$

By Lemma~\ref{delta}, the first part is greater than~$c\langle T_{\delta}F,F\rangle=c\int_M|f(w)|^2d\delta$.

If~$z\in M_{s'}$,
\begin{eqnarray*}
I(z)&=&\int_{p_z^{-1}(D(z,\frac{3}{2}R)\cap TM|_z)}f(w)\frac{(1-|w|^2)^{n-d}}{(1-\langle z,w\rangle)^{n+1}}dv_d(w)\\
&=&\int_{D(z,\frac{3}{2}R)\cap TM|_z}fp_z^{-1}(\eta)\frac{(1-|p_z^{-1}(\eta)|^2)^{n-d}}{(1-\langle z,\eta\rangle)^{n+1}}\frac{E(w')}{E(z')}dv(\eta)\\
&=&\int_{D(z,\frac{3}{2}R)\cap TM|_z}fp_z^{-1}(\eta)\frac{(1-|\eta|^2)^{n-d}}{(1-\langle z,\eta\rangle)^{n+1}}g(\eta)dv(\eta)
\end{eqnarray*}
where
$$g(\eta)=\frac{(1-|p_z^{-1}(\eta)|^2)^{n-d}}{(1-|\eta|^2)^{n-d}}\frac{E(w')}{E(z')}.$$
By Lemma~\ref{s2}~(3) and the absolute continuity of~$E$, we could enlarge~$s$~(so that the Euclidian size of~$D(z,2R)$~is small enough) such that~$g(\eta)$~is sufficiently close to~$1$~and
$$|g(\eta)-1|\leq\epsilon g(\eta).$$
By Lemma~\ref{integral on affine space},
$$\int_{D(z,\frac{3}{2}R)\cap TM|_z}fp_z^{-1}(\eta)\frac{(1-|\eta|^2)^{n-d}}{(1-\langle z,\eta\rangle)^{n+1}}dv(\eta)=C_zf(z),$$
where~$C_z\geq C_{\frac{3}{2}R}$. And
\begin{eqnarray*}
&&|\int_{D(z,\frac{3}{2}R)\cap TM|z}fp_z^{-1}(\eta)\frac{(1-|\eta|^2)^{n-d}}{(1-\langle z,\eta\rangle)^{n+1}}(g(\eta)-1)dv(\eta)|\\
&\leq&\epsilon\int_{D(z,\frac{3}{2}R)\cap TM|_z}|fp_z^{-1}(\eta)|\frac{(1-|\eta|^2)^{n-d}}{|1-\langle z,\eta\rangle|^{n+1}}g(\eta)dv(\eta)\\
&\leq&\epsilon\int_{M_{s}}|f(w)|\frac{(1-|w|^2)^{n-d}}{|1-\langle z,w\rangle|^{n+1}}dv_d(w)
\end{eqnarray*}

So
\begin{eqnarray*}
&&\int_{M_{s'}}|I(z)|^2(1-|z|^2)^{n-d}dv_d(z)\\
&\geq&\frac{1}{2}C_{\frac{3}{2}R}^2\int_{M_{s'}}|f(z)|^2(1-|z|^2)^{n-d}dv_d(z)\\
& &-\epsilon^2\int_{M_{s'}}\bigg(\int_{M_s}|f(w)|\frac{(1-|w|^2)^{n-d}}{|1-\langle z,w\rangle|^{n+1}}dv_d(w)\bigg)^2(1-|z|^2)^{n-d}dv_d(z)
\end{eqnarray*}
Using Holder's inequality and Lemma~\ref{s3} (1), the second part is smaller than
\begin{eqnarray*}
&&\epsilon^2\int_{M_{s'}}\bigg(\int_{M_s}\frac{(1-|w|^2)^{\frac{n-d}{2}}(1-|z|^2)^{\frac{n-d}{2}}}{|1-\langle z,w\rangle|^{n+1}}dv_d(w)\bigg)\\
& &~~\cdot\bigg(\int_{M_s}|f(w)|^2\frac{(1-|w|^2)^{\frac{3(n-d)}{2}}(1-|z|^2)^{\frac{n-d}{2}}}{|1-\langle z,w\rangle|^{n+1}}dv_d(w)\bigg)dv_d(z)\\
&\leq&C\epsilon^2\int_{M_s}\int_{M_{s'}}\frac{(1-|w|^2)^{\frac{n-d}{2}}(1-|z|^2)^{\frac{n-d}{2}}}{|1-\langle z,w\rangle|^{n+1}}dv_d(z)\\
&&~~~~~~~|f(w)|^2(1-|w|^2)^{n-d}dv_d(w)\\
&\leq&C^2\epsilon^2\int_{M_s}|f(w)|^2(1-|w|^2)^{n-d}dv_d(w)\\
&\leq&C^2\epsilon^2\int_M|f|^2d\mu_s.
\end{eqnarray*}
The above estimation is inspired from \cite{JMT}. We will use the same kind of argument in the estimation of $II(z)$.
Combining the above, we have
$$\int_M|I(z)|^2d\mu_{s'}\geq C_1\int_M|f|^2d\mu_{s'}-C_2\epsilon^2\int_M|f|^2d\mu_s.$$

Next we estimate~$II(z)$.
\begin{eqnarray*}
&&\int_M|II(z)|^2d\mu_{s'}(z)\\
&=&\sum_{i=1}^m\bigg|\int_{M_s}f(w)\frac{(1-|w|^2)^{n-d}}{(1-\langle z_i,w\rangle)^{n+1}}dv_d(w)\bigg|^2(1-|z_i|^2)^{n+1}+\int_{M_{s'}}\bigg|T_{\delta}F(z)+\\
& &\int_{M_s\backslash p_z^{-1}(D(z,\frac{3}{2}R)\cap TM|_z)}f(w)\frac{(1-|w|^2)^{n-d}}{(1-\langle z,w\rangle)^{n+1}}dv_d(w)\bigg|^2(1-|z|^2)^{n-d}dv_d(z)\\
&\leq&A+2B+2C.
\end{eqnarray*}
Where
$$A=\sum_{i=1}^m\bigg|\int_{M_s}f(w)\frac{(1-|w|^2)^{n-d}}{(1-\langle z_i,w\rangle)^{n+1}}dv_d(w)\bigg|^2(1-|z_i|^2)^{n+1},$$
$$B=\int_{M_{s'}}|T_{\delta}f(z)|^2(1-|z|^2)^{n-d}dv_d(z),$$
and
$$C=\int_{M_{s'}}\bigg|\int_{M_s\backslash p_z^{-1}(D(z,\frac{3}{2}R)\cap TM|_z)}f(w)\frac{(1-|w|^2)^{n-d}}{(1-\langle z,w\rangle)^{n+1}}dv_d(w)\bigg|^2(1-|z|^2)^{n-d}dv_d(z).$$

Let $a=d(\Sigma,M_s)$, we have

\begin{eqnarray*}
A&\leq&(1/2a^2)^{-2(n+1)}\sum_{i=1}^m(1-|z_i|^2)^{n+1}\bigg(\int_{M_s}|f(w)|(1-|w|^2)^{n-d}dv_d(w)\bigg)^2\\
&\leq&C\bigg(\int_{M_s}|f(w)|^2(1-|w|^2)^{n-d}dv_d(w)\bigg)\bigg(\int_{M_s}(1-|w|^2)^{n-d}dv_d(w)\bigg)\\
&\leq&C(1-s^2)^{n-d}\int_M|f(w)|^2d\mu_s(w)
\end{eqnarray*}
where the first inequality is because
$$
|1-\langle z_i,w\rangle|\geq 1-Re\langle z_i,w\rangle\geq1/2(|z_i|^2+|w|^2-Re\langle z_i,w\rangle)=1/2|z_i-w|^2
$$
and the second inequality is by Holder's inequality.
By taking $s$ closer to $1$, we could make $$A\leq\epsilon^2\int_M|f(w)|^2d\mu_s(w).$$

Similar argument will give us
$$
B\leq\epsilon^2\int_M|f(w)|^2d\mu_s(w).
$$

Now we estimate $C$.
\begin{eqnarray*}
C&\leq&\int_{M_{s'}}\bigg|\int_{M_s\backslash D(z,R)}f(w)\frac{(1-|w|^2)^{n-d}}{(1-\langle z,w\rangle)^{n+1}}dv_d(w)\bigg|^2(1-|z|^2)^{n-d}dv_d(z)\\
&\leq&\int_{M_{s'}}\bigg(\int_{M_s\backslash D(z,R)}\frac{(1-|w|^2)^{\frac{n-d}{2}}(1-|z|^2)^{\frac{n-d}{2}}}{|1-\langle z,w\rangle|^{n+1}}dv_d(w)\bigg)\\
& &\cdot\bigg(\int_{M_s\backslash D(z,R)}|f(w)|^2\frac{(1-|w|^2)^{\frac{3(n-d)}{2}}}{|1-\langle z,w\rangle|^{n+1}}dv_d(w)\bigg)(1-|z|^2)^{\frac{n-d}{2}}dv_d(z)\\
&\leq&\epsilon\int_{M_s}\bigg(\int_{M_{s'}\backslash D(w,R)}\frac{(1-|w|^2)^{\frac{n-d}{2}}(1-|z|^2)^{\frac{n-d}{2}}}{|1-\langle z,w\rangle|^{n+1}}dv_d(z)\bigg)\\
& &\cdot|f(w)|^2(1-|w|^2)^{n-d}dv_d(w)\\
&\leq&\epsilon^2\int_{M_s}|f(w)|^2(1-|w|^2)^{n-d}dv_d(w).
\end{eqnarray*}

Combining the three inequalities, we get
$$\int_M|II(z)|^2d\mu_{s'}(z)\leq5\epsilon^2\int_M|f|^2d\mu_s.$$
Finally, we have
\begin{eqnarray*}
\int_M|\tTmus f(z)|^2d\mu_s(z)&\geq&\int_M|\tTmus f(z)|^2d\mu_{s'}(z)\\
&\geq&\frac{1}{2}\int_M|I(z)|^2d\mu_{s'}(z)-\int_M|II(z)|^2d\mu_{s'}(z)\\
&\geq&C\int_M|f|^2d\mu_{s'}-C'\epsilon^2\int_M|f|^2d\mu_s
\end{eqnarray*}
This holds for all $f\in\Ps$. From the argument in the beginning, we can find a finite codimensional space~$L\subset \mathcal{P}(\mu_s)$~such that~$\forall f\in L$, $$\int_M|f|^2d\mu_{s'}>\frac{1}{2}\int_M|f|^2d\mu_s.$$
Therefore $\forall f\in L$,
$$
\int_M|\tTmus f(z)|^2d\mu_s(z)\geq(\frac{1}{2}C-C'\epsilon^2)\int_M|f|^2d\mu_s
$$
Take~$\epsilon>0$~such that~$\alpha=\frac{1}{2}C-C'\epsilon^2>0$. Then
$$
\|\tTmus f\|_{\mu_s}^2\geq\alpha\|f\|_{\mu_s}^2,~~~\forall f\in L.
$$
Next we show that $\ker\tTmus=\{0\}$. Consider the commuting diagram
$$
\xymatrix{
Q\ar[r]^{\Tmus}\ar[d]^{R}&Q\ar[d]^{R}\\
\Ps\ar[r]^{\tTmus}&\Ps
}
$$
Since $\tTmus$ is positive, it suffices to show that $Range\tTmus$ is dense in $\Ps$. We already know that $Range\Tmus$ is dense in $Q$( since $\ker\Tmus=\{0\}$). Therefore $R\Tmus(Q)$ is dense in $R(Q)$, which is dense in $\Ps$. So $Range\tTmus\supset R\Tmus(Q)$ is dense in $\Ps$. Hence $\ker\tTmus=\{0\}$.

Now suppose $\tTmus$ is not bounded below, then there exists a pairwise orthogonal sequence $\{f_n\}\subset\Ps$, $\|f_n\|_{\mu_s}=1$ such that $\|\tTmus(f_n)\|_{\mu_s}\to0$, $n\to\infty$. Since $L$ is finite codimensional,
$$
\|f_n-Lf_n\|_{\mu_s}\to0, n\to\infty.
$$
But
$$
\|\tTmus Lf_n\|_{\mu_s}\leq\|\tTmus f_n\|_{\mu_s}+\|\tTmus(f_n-Lf_n)\|_{\mu_s}\to0,~~~n\to\infty,
$$
a contradiction. So $\tTmus$ is bounded below.

Suppose
$$\|\tTmus f\|_{\mu_s}^2\geq c\|f\|_{\mu_s}^2,~~~\forall f\in\Ps,$$
then $\forall F\in\ber$,
$$
\langle\Tmus^3 F, F\rangle=\|\tTmus RF\|_{\mu_s}^2\geq c^2\|RF\|_{\mu_s}^2=c^2\langle\Tmus F, F\rangle.
$$
This means $\Tmus^3\geq c^2\Tmus$, which implies that $\Tmus$ is bounded below on $Q$. Therefore $\|\dot\|_{\mu_s}$ and $\|\dot\|$ are equivalent on $Q$. This completes the proof when $\tilde{M}$ is connected.

If $\tilde{M}$ is not connected, then by the theorem in \cite[Page 52]{Complex analytic sets}, after restricting it to a smaller neighborhood of $\clb$, we can divide $\tilde{M}$ into finitely many connected components, each two having positive Euclidian distance(although they may have different dimensions). Then we divide $II(z)$ into more parts, the rest of the proof remains unchanged.
\end{proof}

\begin{cor}\label{toe}
Suppose $M$ satisfies the properties of Theorem \ref{manifold}, then the projection operator onto $P$ is in the Toeplitz algebra $\toe$.
\end{cor}



\begin{rem}\label{rem123}
\begin{itemize}
\item[(1)] Set $d\nu=(1-|w|^2)^{n-d}dv_d|_{M_s}$, then $T_{\nu}$ is a finite rank perturbation of $T_{\mu_s}$ and $\ker T_{\nu}=P$. Restricting to $Q$, $T_{\mu_s}$ is invertible and $T_{\nu}$ is a finite rank perturbation. So $T_{\nu}$ has index $0$. Since $T_{\nu}$ doesn't vanish on $Q$, it is invertible. Therefore $0$ is also isolated in $\sigma(T_{\nu})$, i.e., the measure $\nu$ also satisfies the hypotheses of Theorem \ref{main}. But how to prove it directly is still a problem. Also, for any positive measure $\mu$ on $M$ that is greater than $\nu$, since $T_{\mu}\geq T_{\nu}$ and they have the same kernel, $\mu$ also satisfies the assumptions in Theorem \ref{main}. In particular, we can choose $\mu$ to be the measure that defines the weighted Bergman space on $M$(cf. \cite{DYT}). For measures $\mu$ satisfying the hypotheses of Theorem \ref{main}, the spaces $\mathcal{P}(\mu)$ are the same(with equivalent norms). Sometimes we write $\mathcal{P}$ for simplicity.
\item[(2)] By Proposition 4.4 in \cite{DYT}, the extensions associated to the quotient module and the module $\mathcal{P}$ are unitarily equivalent. If we can prove that the analytic polynomials are dense in the weighted Bergman space $L_{a,n-d}^2(M)$ defined in \cite{DYT}, i.e., $\mathcal{P}=L_{a,n-d}^2(M)$, then the quotient module also defines a K-homology class of the boundary $\tilde{M}\cap\partial\bn$, which we expect to be the fundamental class of $\tilde{M}\cap\partial\bn$ defined by the CR-structure on it.
\item[(3)] Now that we already know the projection $Q$ is in the Toeplitz algebra, there is an easier way to check if a measure $\mu$ satisfy the hypotheses of Theorem \ref{main}. For example, suppose we know that
    $$\|{T_{\mu}}_x-Q_x\|\leq a<1,~~\forall x\in M_{\mathcal{A}}/\bn,$$
   where $S_x$ for an operator $S$ is defined in \cite{Suarez07}, then from Theorem 10.1 in \cite{Suarez07}, it's easy to prove that $T_{\mu}$ is Fredholm as an operator on $Q$. Therefore $T_{\mu}$ is invertible whenever $T_{\mu}$ doesn't vanish on $Q$.
\end{itemize}
\end{rem}
\begin{rem}\label{remext2}
By the previous remark (1), the proof of Theorem \ref{manifold} gives an extension operator from $\mathcal{P}(\mu_s)$ to $\ber$. This is a generalization of Theorem 4.3 in \cite{DYT}. We state it as a corollary.
\begin{cor}\label{extension}
Suppose $\tilde{M}$ is a $d$-dimensional complex analytic subset of an open neighborhood of $\clb$ which intersects $\pbn$ transversally and has no singular point on $\pbn$. Let $M=\tilde{M}\cap\bn$, $\mu=(1-|w|^2)^{n-d}dv_d$ on $M$ and $\mathcal{P}$ be the closed subspace of $L^2(\mu)$ generated by analytic polynomials. Then $\mathcal{P}$ is exactly the range of the restriction operator
$$
R: \ber\to L^2(\mu), f\mapsto f|_M.
$$
Equivalently, there is an extension operator
$$
E: \mathcal{P}\to\ber,
$$
which is a right inverse of $R$.
\end{cor}

In the case of \cite{Beatrous} and \cite{DYT}, $\mathcal{P}(\mu_s)$ coincides with the weighted Bergman space on $M$. We don't know whether the two spaces are the same in general.

The papers \cite{Beatrous} and \cite{DYT} prove the existence of a extension map by giving an integral formula directly and proving the boundedness of it. In fact, our proof also shows the existence of an integral formula defining the extension operator $E$. $\forall f\in\mathcal{P}$, $\forall z\in\bn$,
$$
Ef(z)=\langle Ef, K_z\rangle=\langle f, E^*K_z\rangle_{\mathcal{P}}=\int_M f(w)\overline{E^*K_z(w)}d\mu(w).
$$
This is a new result in the theory of holomorphic extension.
\end{rem}

\section{$p$-Essential Normality for $p>2d$}\label{secpess}

This section is an attempt to solving the unweakened Geometric Arveson-Douglas Conjecture. Throughout this section, we assume $M$ satisfies the hypothesis of Theorem~\ref{manifold}. Let $\mu=(1-|w|^2)^{n-d}dv_d$, where $d$ is the complex dimension of $M$. Then from Remark \ref{rem123} above, $\mu$ satisfies the hypothesis of Theorem~\ref{main}. The restriction operator $R$ is one-to-one when restricted to the quotient space $Q$. In this section, we will show that the range space $\mathcal{P}\subset L^2(\mu)$ is $p$-essentially normal( $p>d$) as a Hilbert module. Moreover, we will show that for $p>2d$, the two modules are ``equivalent'' modulo $\mathcal{S}^p$, in particular, $Q$ is $p$-essentially normal for $p>2d$.

The following lemma is a generalization of Lemma 7 in \cite{hankel}.
\begin{lem}
Suppose $2\leq p<+\infty$ and $G(z,w)$ is $\mu$-measurable in both variables. Let $A_G$ be the integral operator on $L^2(\mu)$ defined by
$$
A_Gf(z)=\int_M G(z,w)K_w(z)f(w)d\mu(w).
$$
If
$$
\int_M\int_M|G(z,w)|^p|K_w(z)|^2d\mu(z)d\mu(w)<+\infty,
$$
then $A_G$ is in the Schatten $p$ class $\mathcal{S}^p$.
\end{lem}

\begin{proof}
When $p=2$ the result is well-known. When $|G(z,w)|$ is bounded, let $h(z)=(1-|z|^2)^{-1/4}$, then
\begin{eqnarray*}
\int_M|G(z,w)K_w(z)|h(z)d\mu(z)&\leq&C\int_M\frac{1}{|1-\langle z,w\rangle|^{n+1}(1-|z|^2)^{1/4}}d\mu(z)\\
&\leq&C\int_{\bn}\frac{1}{|1-\langle z,w\rangle|^{n+1}(1-|z|^2)^{1/4}}dv(z)\\
&\leq&C(1-|w|^2)^{-1/4}.
\end{eqnarray*}
Similarly,
$$
\int_M|G(z,w)K_w(z)|h(w)d\mu(w)\leq Ch(z).
$$
By Schur's test, $A_G$ defines a bounded operator. Now let $d\nu$ be the measure $|K_w(z)|^2d\mu(z)d\mu(w)$ on $M\times M$ and consider the map
$$
L^2(\nu)+L^{\infty}(\nu)\to\mathcal{B}(L^2(\mu)),~G(z,w)\to A_G.
$$
Then by non-commutative interpolation(cf. \cite{interpolation} ), the lemma is true.
\end{proof}

\begin{lem}\label{hatT}
Let
$$
\hat{T}_{\mu}: L^2(\mu)\to L^2(\mu),~\hat{T}_{\mu}f(z)=\int_M f(w)K_w(z)d\mu(w),
$$
$$
\hat{M}_{z_i}: L^2(\mu)\to L^2(\mu),~f\mapsto z_if.
$$
Then the commutator $[\hat{M}_{z_i},\hat{T}_{\mu}]\in\mathcal{S}^p$, $\forall p>2d$.
\end{lem}

\begin{proof}
It's easy to check that
$$
[\hat{T}_{\mu},\hat{M}_{z_i}]f(z)=\int_M(w_i-z_i)K_w(z)f(w)d\mu(w),~\forall z\in M,~\forall f\in L^2(\mu).
$$
Let $G(z,w)=w_i-z_i$ in the last lemma, then it suffices to prove that
$$
\int_M\int_M|z-w|^p|K_w(z)|^2d\mu(w)d\mu(z)<+\infty,~\forall p>2d.
$$
Now when $|z-w|>\delta$ the intergrade is bounded. On the other hand, using the same technique in the last section,
\begin{eqnarray*}
&&\int_{B(z,\delta)\cap M}|z-w|^p|K_w(z)|^2d\mu(w)\\
&\leq& C\int_{\mathbb{B}_d}|z'-w'|^p\frac{(1-|w'|^2)^{n-d}}{|1-\langle z',w'\rangle|^{2(n+1)}}dv_d(w')\\
&\leq&C\int_{\mathbb{B}_d}\frac{(1-|w'|^2)^{n-d}}{|1-\langle z',w'\rangle|^{2(n+1)-p/2}}\\
&\leq&C(1-|z|^2)^{-(n+1-p/2)}.
\end{eqnarray*}
Therefore
\begin{eqnarray*}
&&\int_M\int_M|z-w|^p|K_w(z)|^2d\mu(w)d\mu(z)\\
&\leq&C\int_M(1-|z|^2)^{-(n+1-p/2)}d\mu(z)\\
&=&C\int_M(1-|z|^2)^{p/2-d-1}dv_d(z).
\end{eqnarray*}
Using local coordinates, the last integral is less than
$$
C\int_{\mathbb{B}_d}(1-|z|^2)^{p/2-d-1}dv(z),
$$
which is bounded when $p>2d$. This completes the prove.
\end{proof}

\begin{lem}\label{P is p essential normal}
The module $\mathcal{P}$ is $p$-essentially normal for all $p>d$. That is $[M_{z_i},M_{z_j}^*]\in\mathcal{S}^p$, $p>d$, where $M_{z_i}$ are the multiplication operators on $\mathcal{P}$.
\end{lem}

\begin{proof}
Since the module action on $L^2(\mu)$ is normal, from Proposition 4.1 in \cite{Arv p summable}, it suffices to show that $[\mathcal{P}, \hat{M}_{z_i}]\in\mathcal{S}^p$, $p>2d$. Clearly $\hat{T}_{\mu}$ is self-adjoint and $\mathcal{P}$ is a $C^{\infty}$-functional calculous of $\hat{T}_{\mu}$. Combining lemma \ref{hatT} and Proposition 5 in Appendix I of \cite{connes} one gets the desired result.
\end{proof}

Now we are ready to prove the main theorem of this section.
\begin{thm}\label{2d ess nor}
Under the assumptions of Theorem \ref{manifold}, the quotient module $Q$ is $p$-essentially normal for all $p>2d$.
\end{thm}

\begin{proof}
Consider the following commuting graph:
$$
\xymatrix{
Q\ar[r]^{S_{z_i}}\ar[d]_R&Q\ar[d]^R\\
\mathcal{P}\ar[r]^{M_{z_i}}&\mathcal{P}
}
$$
Then $S_{z_i}=R^{-1}M_{z_i}R$. Therefore
$$
[S_{z_i},S_{z_j}^*]=R^{-1}M_{z_i}RR^*M_{z_j}^*R^{*-1}-R^*M_{z_j}^*(RR^*)^{-1}M_{z_i}R.
$$
On the other hand, for any $f$, $g\in Q$,
$$
\langle Rf,Rg\rangle=\int_Mf(z)\overline{g(z)}d\mu(z)=\langle T_{\mu}f,g\rangle.
$$
Hence $R^*R=T_{\mu}$. Also, from the following commuting graph,
$$
\xymatrix{
Q\ar[r]^{T_{\mu}}\ar[d]_R&Q\ar[d]^R\\
\mathcal{P}\ar[r]^{\tilde{T}_{\mu}}&\mathcal{P}
}
$$
$\tilde{T}_{\mu}=RT_{\mu}R^{-1}=RR^*RR^{-1}=RR^*$. From lemma \ref{P is p essential normal}, $[\tilde{T}_{\mu}, M_{z_i}]=[\hat{T}_{\mu},\hat{M}_{z_i}]|_{\mathcal{P}}\in\mathcal{S}^p$, $p>2d$. Therefore
\begin{eqnarray*}
&&[S_{z_i},S_{z_j}^*]\\
&=&R^{-1}M_{z_i}\tilde{T}_{\mu}M_{z_j}^*R^{*-1}-R^*M_{z_j}^*\tilde{T}_{\mu}^{-1}M_{z_i}R\\
&=&R^*M_{z_i}M_{z_j}^*R^{*-1}-R^*M_{z_j}^*M_{z_i}R^{*-1}~~~~(\mbox{modulo }\mathcal{S}^p)\\
&=&R^*[M_{z_i},M_{z_j}^*]R^{*-1}\in\mathcal{S}^p
\end{eqnarray*}
for $p>2d$. This completes the proof.
\end{proof}

\begin{rem}
Although our proof doesn't give $p$-essential normality for $p>d$, we still think it is true in our case. In the case of a hyperplane, the operators $[\hat{T}_{\mu}, \hat{M}_{z_i}]|_{\mathcal{P}}$ are $0$, therefore there are no obstructions in this case. It's possible that after modifying the measure and using more detailed estimation, one can show that the above operators are in $\mathcal{S}^p$, $p>d$, which would imply the unweakened Geometric Arveson-Douglas Conjecture.
\end{rem}

\section{Arveson-Douglas Conjecture}\label{arvsec}

Suppose $I\subset\mathbb{C}[z_1,\cdots,z_n]$ is a polynomial ideal. Let $\tilde{M}=Z(I)$, the zero variety of $I$. If $\tilde{M}$ satisfies the hypotheses of Theorem \ref{manifold}, then we know that the quotient module $Q$ is $p$-essential normal, $\forall p>2d$. In this section, we show that when $I$ is radical, $[I]=P$. Here $[I]$ is the closure of $I$ in $\ber$. Therefore we prove the weak Arveson-Douglas Conjecture for such ideals.

\begin{thm}\label{arv}
Suppose $I\subset\mathbb{C}[z_1,\cdots,z_n]$ is a radical polynomial ideal. If $\tilde{M}:=Z(I)$ has no singular point on $\pbn$ and intersects $\pbn$ transversally, then $[I]=\{f\in\ber: f|_M=0\}$. Here $M=\tilde{M}\cap\bn$. As a consequence, the quotient module $Q=[I]^{\perp}$ is $p$-essentially normal, $p>2d$. Moreover, the projection operator onto $Q$ is in the Toeplitz algebra $\toe$.
\end{thm}

To prove the above theorem, we will show that the extension operators corresponding to the subvariety $1/t\tilde{M}$ for $t$ in a small interval $[1,t_0]$ are uniformly bounded.

In general, suppose $\tilde{M}$ satisfies the assumptions of Theorem \ref{manifold}. Then for $t$ close enough to $1$,  $1/t\tilde{M}$ also satisfies the hypotheses of Theorem \ref{manifold}. Equivalently, there is an extension operator from $M^t$ to $t\bn$. In the proof of Theorem \ref{manifold}, the number $\epsilon$ depends on $\tilde{M}$, $R$ and $s$ depend on $\epsilon$ and $\tilde{M}$ and $s'$ depends on $s$ and $R$. By taking a larger $R$, we can find $0<s<s'<1$ and $t_1>1$ such that $\forall 1\leq t\leq t_1$, the measure
$$
d\mu_s^t=\sum_{i=1}^m(t^2-|z_i|^2)^{n+1}\delta_{z_i}+(t^2-|w|^2)^{n-d}dv_d|_{M^t}
$$
is suitable for the proof of Theorem \ref{manifold}. That means, suppose $f\in L_a^2(t\bn)$, $$\int_{M_s^{s'}}|f|^2d\mu_s^t\leq1/2\int_M|f|^2d\mu_s^t,$$
then
$$\langle T_{\mu_s^t}^3f,f\rangle\geq\alpha\langle T_{\mu_s^t}f,f\rangle$$
for some $\alpha>0$ that doesn't depend on $t$.
We remind the reader that there is in fact a normalizing constant between the spaces corresponding $1/t\tilde{M}$ and $\tilde{M}$, but since it tends to $1$ as $t$ tends to $1$, we omit the difference.

Let $\mathcal{P}_t$, $Q_t$, $T_{\mu_s^t}$ and $R_t$ be the obvious ones. Then $\mathcal{P}_t=Range R_t$. Define
$$
E_t: \mathcal{P}_t\to Q_t, f|_{M^t}\mapsto f\in Q_t,
$$
$$
\tilde{T}_{\mu_s^t}:\mathcal{P}_t\to\mathcal{P}_t, f\mapsto R_tT_{\mu_s^t}E_tf.
$$
Then $\tilde{T}_{\mu_s^t}$ are uniformly bounded for all $t$.
Let $L_1\subset\mathcal{P}_1$ be the finite co-dimensional subspace that $\forall f\in L_1$,
$$
\int_{M_s^{s'}}|f|^2d\mu_s^1\leq1/3\int_M^1|f|^2d\mu_s^1.
$$
Then $\forall f\in R_1^{-1}L_1\subset Q_1$,
$$
\langle T_{\mu_s^1}^3f,f\rangle\geq\alpha\langle T_{\mu_s^1}f,f\rangle\geq\frac{\alpha}{\|T_{\mu_s^1}\|}\|T_{\mu_s^1}f\|^2.
$$
So
$$
\int_{M^1}|T_{\mu_s^1}f|^2d\mu_s^1\geq c\|T_{\mu_s^1}f\|^2.
$$
This means for every function $g\in\mathcal{L}_1:=\tilde{T}_{\mu_s^1}L_1$,
$$
\|E_1g\|\leq C\|g\|_{\mu_s^1}.
$$
There is a natural inclusion $\mathcal{P}_t\subset\mathcal{P}_{t'}, t'\leq t$, given by restriction. In particular, $\mathcal{P}_t\subset\mathcal{P}_1$, $\forall t\geq1$. Let $L_t=L_1\cap\mathcal{P}_t$. Then for $t$ close enough to $1$, $\forall f\in L_t$,
$$
\int_{M_s^{s'}}|f|^2d\mu_s^t\leq1/2\int_{M^t}|f|^2d\mu_s^t.
$$
So $\forall g\in\mathcal{L}_t:=\tilde{T}_{\mu_s^t}L_t$,
$$
\|E_tg\|\leq C\|g\|_{\mu_s^t}.
$$
The spaces $\mathcal{L}_t\subset\mathcal{P}_t$ have the same co-dimension: since the polynomials are dense, we can find finite dimensional space $N$ consisting of(restriction of) polynomials such that $N+L_1=\mathcal{P}_1$, $N\cap L_1=\{0\}$. Then it's easy to prove that $N+L_t=\mathcal{P}_t$, $N\cap L_t=\emptyset$. Since the operators $\tilde{T}_{\mu_s^t}$ are one to one and surjective, codim$\mathcal{L}_t=$ codim$L_t=\dim N$.

We denote the norms $\|\cdot\|_{\mu_s^t}$ by $\|\cdot\|_t$ from this point.
Let $\eta>0$ be determined later. We can take finite dimensional space $\mathcal{N}$ consisting of polynomials such that $\mathcal{N}+\mathcal{L}_1=\mathcal{P}_1$ and $\forall f\in\mathcal{N}, \forall g\in\mathcal{L}_1$,
$$
\frac{|\langle f,g\rangle_1|}{\|f\|_1\|g\|_1}<\eta.
$$

\textbf{Claim}: $\exists t_0>1$, such that $\forall 1\leq t\leq t_0$, $\forall f\in\mathcal{N}$, $\forall g\in\mathcal{L}_t$,
$$
\frac{|\langle f,g\rangle_t|}{\|f\|_t\|g\|_t}<1/2.
$$
Suppose the claim is not true, then there exists a sequence $t_n\to1$, $f_n\in\mathcal{N}, g_n\in\mathcal{L}_{t_n}$, $\|f_n\|_{t_n}=\|g_n\|_{t_n}=1$, such that
$$
\langle f_n,g_n\rangle_{t_n}\geq1/2.
$$
Then $g_n=\tilde{T}_{\mu_s^{t_n}}h_n$, $h_n\in L_{t_n}$, by previous discussion,
$$
\|h_n\|_{t_n}\leq C\|g_n\|_{t_n}=C
$$
for some $C>0$. It's easy to show that the norms $\|\cdot\|_t$ are uniformly equivalent on the space $\mathcal{N}$. So $f_n$(has a subsequence) tends to some $f\in\mathcal{N}$ uniformly on all $\|\cdot\|_{t_n}$ norms. Hence for sufficiently large $n$ we have
$$
|\langle f,g_n\rangle_{t_n}|\geq1/3.
$$
We also have $\|f\|_1=1$. So
$$
|\langle \tilde{T}_{\mu_s^{t_n}}f,h_n\rangle_{t_n}|=|\langle f,g_n\rangle_{t_n}|\geq1/3.
$$
Finally, we prove that
$$
|\langle\tilde{T}_{\mu_s^{t_n}}f,h_n\rangle_{t_n}-\langle\tilde{T}_{\mu_s^1}f,h_n\rangle_1|\to0, n\to\infty.
$$
Since
$$|\langle\tilde{T}_{\mu_s^1}f,h_n\rangle_1|=|\langle f,\tilde{T}_{\mu_s^1}h_n\rangle_1|<\eta\|f\|_1\|\tilde{T}_{\mu_s^1}h_n\|_1\leq C\eta,$$
Take $\eta$ such that $C\eta\leq1/4$, this proves the claim
\begin{eqnarray*}
|\tilde{T}_{\mu_s^{t_n}}f(z)|&\leq&|\int_{M_s^{t_n}}f(w)\frac{(t_n^2-|w|^2)^{n-d}}{(t_n^2-\langle z,w\rangle)^{n+1}}dv_d(w)|\\
& &+|\sum_{i=1}^mf(z_i)\frac{(t_n^2-|z_i|^2)^{n+1}}{(t_n^2-\langle z,z_i\rangle)^{n+1}}|\\
&\leq&C\int_{M_s^{t_n}}\frac{(t_n^2-|w|^2)^{n-d}}{|t_n^2-\langle z,w\rangle|^{n+1}}dv_d(w)+C\\
&\leq&C+C\int_{M_s^{t_n}\cap B(z,\delta)}\frac{(t_n^2-|w|^2)^{n-d}}{|t_n^2-\langle z,w\rangle|^{n+1}}dv_d(w)\\
&\leq&C+C\int_{t_n\mathbb{B}_d}\frac{(t_n^2-|w'|^2)^{n-d}}{|t_n^2-\langle z',w'\rangle|^{n+1}}dv(w')\\
&\leq&C+C\log\frac{1}{1-|z'/t_n|^2}\\
&\leq&C+C\log\frac{1}{t_n^2-|z|^2}.
\end{eqnarray*}
The above estimation is similar with those used in the last section, we omit the details. The third inequality from the bottom is from Lemma \ref{1.4.10 Rudin}.
\begin{eqnarray*}
&&\langle\tilde{T}_{\mu_s^{t_n}}f,h_n\rangle_{t_n}-\langle\tilde{T}_{\mu_s^1}f,h_n\rangle_1\\
&=&\int_{M^{t_n}}\tilde{T}_{\mu_s^{t_n}}f(z)\overline{h_n(z)}d\mu_s^{t_n}(z)-\int_{M^1}\tilde{T}_{\mu_s^1}f(z)\overline{h_n(z)}d\mu_s^1(z)\\
&=&\int_{M_1^{t_n}}\tilde{T}_{\mu_s^{t_n}}f(z)\overline{h_n(z)}d\mu_s^{t_n}(z)\\
& &+\int_{M_s^1}\bigg(\tilde{T}_{\mu_s^{t_n}}f(z)-\tilde{T}_{\mu_s^1}f(z)\frac{(1-|z|^2)^{n-d}}{(t_n^2-|z|^2)^{n-d}}\bigg)\overline{h_n(z)}d\mu_s^{t_n}(z)\\
& &+\sum_{i=1}^m\bigg(\tilde{T}_{\mu_s^{t_n}}f(z_i)(t_n^2-|z_i|^2)^{n+1}-\tilde{T}_{\mu_s^1}f(z_i)(1-|z_i|^2)^{n+1}\bigg)\overline{h_n(z_i)}\\
&=&I_n+II_n+III_n
\end{eqnarray*}
Clearly $III_n\to0, n\to\infty$. For the first part,
\begin{eqnarray*}
&&|\int_{M_1^{t_n}}\tilde{T}_{\mu_s^{t_n}}f(z)\overline{h_n(z)}d\mu_s^{t_n}(z)|\\
&\leq&\bigg(\int_{M_1^{t_n}}|\tilde{T}_{\mu_s^{t_n}}f(z)|^2d\mu_s^{t_n}\bigg)^{1/2}\|h_n\|_{t_n}\\
&\leq&C\bigg(\int_{M_1^{t_n}}(C+C\log\frac{1}{t_n^2-|z|^2})^2(t_n^2-|z|^2)^{n-d}dv_d(z)\bigg)^{1/2}\\
&\leq&Cv_d(M_1^{t_n})\to0.
\end{eqnarray*}
For the second part, since
\begin{eqnarray*}
&&|\tilde{T}_{\mu_s^{t_n}}f(z)-\tilde{T}_{\mu_s^1}f(z)\frac{(1-|z|^2)^{n-d}}{(t_n^2-|z|^2)^{n-d}}|^2(t_n^2-|z|^2)^{n-d}\\
&\leq&2(C+C\log\frac{1}{t_n^2-|z|^2})^2(t_n^2-|z|^2)^{n-d}\\
& &+2(C+C\log\frac{1}{1-|z|^2})^2(1-|z|^2)^{n-d}\frac{(1-|z|^2)^{n-d}}{(t_n^2-|z|^2)^{n-d}}\\
&\leq&C
\end{eqnarray*}
By the dominance convergence theorem,
$$
\int_{M_s^1}|\tilde{T}_{\mu_s^{t_n}}f(z)-\tilde{T}_{\mu_s^1}f(z)\frac{(1-|z|^2)^{n-d}}{(t_n^2-|z|^2)^{n-d}}|^2(t_n^2-|z|^2)^{n-d}dv_d(z)\to0.
$$
Using the holder's inequality, we see that $II_n\to0$. So the proof of claim is complete.

\begin{proof}[\textbf{proof of Theorem \ref{arv}}]
For $\forall f\in\mathcal{N}, \forall g\in\mathcal{L}_t$,
\begin{eqnarray*}
&&\|f+g\|_t^2=\|f\|_t^2+\|g\|_t^2+2Re\langle f,g\rangle_t\\
&\geq&\|f\|_t^2+\|g\|_t^2-\|f\|_t\|g\|_t\geq1/2(\|f\|_t^2+\|g\|_t^2).
\end{eqnarray*}
By continuity, there is a constant $C'>0$ such that $\forall 1\leq t\leq t_0$, $\forall f\in\mathcal{N}$, $\|E_tf\|\leq C'\|f\|_t$. So for any $h\in\mathcal{P}_t$, $h=f+g$, $f\in\mathcal{N}$, $g\in\mathcal{L}_t$,
$$
\|E_t(f+g)\|^2\leq\|E_tf\|^2+\|E_tg\|^2\leq C(\|f\|_t^2+\|g\|_t^2)\leq2C\|f+g\|_t^2.
$$
So the extension operators $E_t$ are uniformly bounded.

Knowing the above result, the proof of Theorem \ref{arv} is exactly the same as in the appendix of \cite{DYT}.
\end{proof}
\section{Summary}
This paper combines ideas from various subject to solve the Geometric Arveson-Douglas Conjecture. First, we view the quotient module as a reproducing kernel Hilbert module on the variety(cf. \cite{Aronszajn}) and seek an equivalent norm given by a measure. The fact that operators of the form $T_{\mu}$ are in the Toeplitz algebra(cf. \cite{Suarez07}) is crucial to our proof. Second, we use the idea that Toeplitz operators are ``localized'' (cf. \cite{localized operator}\cite{Suarez07}) and apply it on the variety, instead of the whole unit ball. Finally, we observe that the fact that the size of the hyperbolic balls tend to $0$ uniformly as the centers tend to the boundary forces the Bergman reproducing kernels $K_z(w)$ to act ``almost'' like reproducing kernels on the quotient space. Using Theorem 4.7 in \cite{Suarez Wick}, we see that these kind of argument also work for the weighted Bergman spaces.

The techniques from complex harmonic analysis reveals some connection between the Geometric Arveson-Douglas Conjecture and the extension of holomorphic functions(cf. \cite{Beatrous}). This idea first appeared in \cite{DYT} and has inspired us to seek connections from different angles.

Under the hypotheses of this paper, quotient modules are closely related to weighted Bergman spaces on the variety. Therefore whatever is true for the Bergman space may also be true on the quotient space.  Moreover, the idea of looking at the operator $T_{\mu}$ ``locally'' on $M$ offers a way to study the Geometric Arveson-Douglas Conjecture for more general varieties.  These relationships will be the future focuses of our research.

Another direction we plan to consider is extending the index result in \cite{DYT}. Recall in \cite{DYT}, a generalization of the Boutet de Monvel result is obtained using the methods in \cite{BDT}. It seems likely we can extend the proofs to cover our case.

We would like to thank Xiang Tang for discussing with us since the early stage of our research, for reading the drafts of this paper and for the valuable suggestions he gave us. We also would like to thank Kai Wang for the valuable discussions over Wechat. The second author would like to thank Kunyu Guo, her advisor in Fudan University, for inviting her to the world of mathematical research and the advice he gave over emails. She also want to thank her fellows in Fudan University, especially Zipeng Wang, for drawing her attention to complex harmonic analysis before her visit to Texas A\&M University.

Ronald G.~Douglas

Texas A\&M University, College Station, TX, 77843

E-mail address: rdouglas@math.tamu.edu\\

Yi Wang

Fudan University, Shanghai, China, 200433

E-mail address: yiwangfdu@gmail.com
\end{document}